\documentclass[12pt]{article}
\usepackage{amsfonts,amssymb,amsthm}
\usepackage{enumitem, amsmath}
\usepackage{latexsym}  
\usepackage{authblk}
\usepackage{hyperref}
\usepackage{mathtools}
\hypersetup{
	colorlinks=true,
	linkcolor=blue,
	citecolor=blue,
	filecolor=green,
	urlcolor=cyan,
	bookmarks=true
}

\makeatletter
\renewcommand*{\eqref}[1]{%
	\hyperref[{#1}]{\textup{\tagform@{\ref*{#1}}}}%
}
\makeatother

\setlist[enumerate,1]{label={\textup{\textbf{\roman*)}}}}

\title{Non-linear Gagliardo--Nirenberg inequality 

involving a second-order elliptic  operator in non-divergent form}
\author{{\large{Agnieszka Ka{\l}amajska}}\\
\textit{Faculty of Mathematics, Informatics and Mechanics, University of Warsaw, Banacha 2, 02-097 Warsaw, Poland; email: A.Kalamajska@mimuw.edu.pl, ORCID: 0000-0001-5674-8059}
\vspace{0.5cm}\\
{\large{Dalimil Pe\v{s}a}}\\
\textit{ Department of Mathematical Analysis, Faculty of Mathematics and Physics,
Charles University, Sokolovsk\'a 83, 186 75 Praha 8, Czech Republic;  email: pesa@karlin.mff.cuni.cz, ORCID: 0000-0001-6638-0913} 
\vspace{0.5cm}\\
{\large{Tom\'a\v{s} Roskovec}}\\
\textit{Faculty of Education, University of South Bohemia in \v{C}esk\'e Bud\v{e}jovice, Jeron\'ymova 10, 371 15 \v{C}esk\'e Bud\v{e}jovice, Czech Republic; email: troskovec@jcu.cz, ORCID: 0000-0003-0438-0066}}

\def\rn{{\mathbf{R}^{n}}}

\usepackage{color}

\newtheorem{Theorem}{\bf Theorem}[section]

\newtheorem{Lemma}[Theorem]{\bf Lemma}
\newtheorem{Remark}[Theorem]{\bf Remark}

\newtheorem{ex}[Theorem]{\bf Example}

\newtheorem{oq}{\bf Open problem}

\makeatletter \@addtoreset{equation}{section}

\makeatother

\newcommand{\beq}{\begin{equation}}
\newcommand{\eeq}{\end{equation}}

\newcommand{\R}{\mathbf{R}}
\newcommand{\Rn}{\mathbf{R}^n}

\newcommand{\A}{\mathbf{A}}

\newcommand{\dee}{\,\textup{d}}

\begin{document}

\maketitle

\begin{flushleft}
\end{flushleft}

 \begin{abstract}
\noindent

We obtain the inequalities of the form $$\int_{\Omega}|\nabla u(x)|^2h(u(x))\dee x\leq C\int_{\Omega} \left( \sqrt{ |P u(x)||{\cal T}_{H}(u(x))|}\right)^{2}h(u(x))\dee x +\Theta,$$ where
$\Omega\subset \rn$ is a bounded Lipschitz domain, $u\in W^{2,1}_{\rm loc}(\Omega)$ is non-negative,   $P$ is a uniformly elliptic operator in non-divergent form, ${\cal T}_{H}(\cdot )$ is certain transformation of the monotone $C^1$ function $H(\cdot)$, which is the primitive of the weight $h(\cdot)$, and $\Theta$ is the boundary term which depends on boundary values of $u$ and $\nabla u$, which hold under some additional assumptions. Our results are linked to some results from probability and potential theories, e.g.~to some variants of the Douglas formulae.
\end{abstract}

\section{Introduction}
\noindent
Let $\Omega\subseteq\R^n$, where $n\ge 2$, be a bounded connected domain with Lipschitz boundary,  let $u: \Omega\rightarrow \R$ be a non-negative function that belongs to the Sobolev space $W^{2,1}_{\rm loc}(\Omega)$.
Further, let $P$ be a uniformly elliptic operator with $C^1$ coefficients up to the
boundary of $\Omega$, given in non-divergent form, and defined by the matrix $\A(x)=\{ a_{i,j}(x)\}_{i,j=1,\dots,n}$, that is,
\begin{equation*}\label{peintro}
Pu(x)=\sum_{i,j}a_{i,j}(x)\frac{\partial^2 u}{\partial x_i\partial x_j}(x)   \text{ a.e. in }\Omega.  \end{equation*}

\noindent
{\it The purpose} of this paper is the derivation and further analysis of identities like
\begin{equation*}\label{id1} 
\begin{aligned}
     \int_{\Omega } \|\nabla u\|_{\A}^{2} h(u(x))\dee x  &= - \int_{\Omega} Pu \;  H(u) \dee x   
     -
    \int_{\Omega } \operatorname{div} \A\cdot \nabla u \;  H (u) \dee x +\Theta,\text{ where }  \\
\Theta &:= \int_{\partial\Omega}  n(x)^T \A(x) \nabla  (\widetilde{H} (u))\dee \sigma (x),    
\end{aligned}
\end{equation*}
as well as the inequalities like
\begin{equation}\label{ine1}  
\begin{aligned}
   \int_{\Omega } \|\nabla u\|_{\A}^{2} h(u)\dee x  
&\le   \int_{\Omega } |Pu||H(u)| \dee x + \Theta \text{ when } \operatorname{div}{\A}\equiv 0,\\
   \int_{\Omega } \|\nabla u\|_{\A}^{2} h(u)\dee x  
&\le  d_{\A} \int_{\Omega }\mathcal{G}_{H}(u)\dee x + 2\int_{\Omega } |Pu||H(u)| \dee x + 2\Theta \text{ otherwise,}  
\end{aligned}
\end{equation}
where $\sigma$ is the $(n-1)$-dimensional Hausdorff measure on $\partial\Omega$, $n(x)$ is the outer normal to $\partial\Omega$ (at point $x \in \partial \Omega$), the norm $\|\cdot\|_\A$ and the constant $d_{\A}$ are related to the operator $P$ as in \ref{(A1)} and \eqref{eq: A norm definition} in Section~\ref{preliminary} below, $H(u)$ and $\widetilde{H}(u)$ are respectively the compositions of $u$
with the first and second-order antiderivative of the function $h(\cdot)$, and $\mathcal{G}_{H}(u)= H^2(u)/h(u)$. See Theorems \ref{LemmaBoundedU} and \ref{LemmaBoundedUcoinfty} for the precise formulations.

\smallskip

As we show in Theorems \ref{structureA} and \ref{simplification2}, the inequalities \eqref{ine1} can be simplified to ones without the term $\int_{\Omega }\mathcal{G}_{H}(u)\dee x$, provided that some additional conditions are satisfied. Moreover, in the proof of Theorem~\ref{simplification2} we exploit Opial-type inequalities that we establish in Theorem~\ref{Opial}:
{ \small
\begin{eqnarray*}
\int_{\Omega \cap \{ 0<u\}} |\mathcal{T}_{H}(u)|^2 h(u)\dee x \lesssim  
\int_{\Omega \cap \{ 0<u\}} \|\nabla u \| |\mathcal{T}_{H}(u)| h(u) \dee x\lesssim 
\int_{\Omega \cap \{ 0<u\}} \|\nabla u\|^2 h(u)\dee x ,
\end{eqnarray*}
}
where $\mathcal{T}_H(u)=H(u)/h(u)$. We believe this result to be of independent interest.

\bigskip
\noindent
{\it Inequalities similar to} \eqref{ine1}  have appeared earlier in the literature in several places. 

Consider the classical Laplace operator $P=\Delta$ and $u:\Omega\rightarrow \R$ being sufficiently regular and satisfying the boundary conditions that ensure $\Theta =0$. Then the inequality \eqref{ine1} follows from a family of inequalities which were obtained earlier in \cite{akitch2}:\begin{eqnarray}\label{laplace}
\int_{\Omega} \|\nabla u\|^{p} h(u) \dee x&\le& C\int_{\Omega} \left(\sqrt{|{\mathcal T}_H(u)\Delta u|}\right)^p h(u) \dee x,
\end{eqnarray}
where ${\mathcal T}_H(u)= H(u)/h(u)$, and $p\ge 2$.

The inequalities \eqref{ine1} were inspired by the following inequality from \cite{kp1}
\begin{eqnarray*}
\int_{\R^n} G(\|\nabla u\|)\dee x&\le& C\int_{\R^n} G(\sqrt{|u|\|\nabla^{(2)}u\|})\dee x,\ u\in C_0^\infty (\R^n), 
\end{eqnarray*}
where $G$ is convex and $G(s)/s^2$ is bounded near $0$, as well as by the earlier work \cite{akitch1}, where in place of $\Delta u$ in \eqref{laplace} one deals with the Hessian  $\nabla^{(2)}u$. The main difference between our approach in this work and the approaches in \cite{akitch1,akitch2} is that we assume that $u\in W^{2,1}_{\rm loc}(\Omega)$ and $\widetilde{H}(u)\in W^{2,1}(\Omega)$, where $\widetilde{H}$ is second antiderivative of $h$, while the assumption in \cite{akitch1,akitch2} was $u\in W^{2,1}(\Omega)\cap C \left ( \overline{\Omega} \right)$, when restricted to $p=2$, see Theorems 3.2 and 3.4 in \cite{akitch1} and Theorem 2.3 in \cite{akitch2}.

Some variants of the inequality \eqref{laplace} in one dimension have been derived and analysed in \cite{cfk, km, akjp, kp2}. As their pioneering source, we consider the inequality due to Mazya from \cite[Lemma 1, Section 8.2.1]{ma}:
\begin{equation*}
\int_{{\rm supp} f'} \left(\frac{ |f'|}{f^{\frac{1}{2}}}\right)^p\dee x \le \left( \frac{p-1}{|1-\frac{1}{2} p|}\right)^{\frac{p}{2}} \int_\R |f''|^{{p}}\dee x,\  \ p>2,  
\end{equation*}
which is valid for all smooth, non-negative, and compactly supported functions $f$ and corresponds to \eqref{laplace} for the choice $h(s)=s^{-\frac{1}{2}}$.

\bigskip

\noindent
Let us observe that we can retrieve a variant of the classical Gagliardo--Nirenberg interpolation inequality from \eqref{ine1}. Indeed, by using the H\"older's inequality, with $1<p,p'<\infty$ that satisfy $1=\frac{1}{p}+\frac{1}{p'}$, and by plugging in $h\equiv 1$ and $H(s)=s$, so that $\mathcal{T}_H(s)=s$, we get

\begin{eqnarray*}\label{clasgn}
\|\nabla u\|_{L^2(\Omega)}\le C\left(\|\nabla^{(2)}u\|^{\frac{1}{2}}_{L^p(\Omega)} \| u\|^{\frac{1}{2}}_{L^{p'}(\Omega)} +\| u\|_{L^{p'}(\Omega)}\right),
\end{eqnarray*}
where $\Omega$ is a bounded and sufficiently regular (we refer to \cite{ga,n1} as the classical sources, as well as to the recent paper \cite{Fiforoso}). Note that $|Pu|\lesssim \|\nabla^{(2)}u\|$ and $\|\nabla u\|^2_\A\approx \|\nabla u\|^2$ by uniform ellipticity of $P$.

Let us note that an interesting non-linear variant of the Gagliardo--Nirenberg interpolation inequality has appeared prior to our research in \cite{RivStrz}. However, neither our theorems nor \cite{akitch1, akitch2} that inspired our research are consequences of this result and we focus on a different setting.

\bigskip
\noindent
{\it Motivations} for our analysis are two-fold:\\
\noindent
A) to obtain a priori estimates for the solutions of non-linear PDEs.
B) applications in harmonic analysis and potential theory; \\
Part A) will be discussed in our forthcoming paper, which will focus on applications of our theory to regularity and qualitative properties  of solutions of second-order elliptic PDEs as e.g.
\begin{equation*}\label{eqEmFtype}Pu=\left\{\begin{array}{ccc}
f(x) u^{-\gamma} & {\rm a. e.~in} &\Omega \\
u\equiv 0 & {\rm on}& \partial\Omega 
\end{array}\right.\text{ for }
\gamma>0, \  f\in L^1(\Omega),  u(x)\in (0,\infty) \ \hbox{\rm  a.e.~in } \Omega.  
\end{equation*}
\smallskip
One of the motivations in the direction B) came from the paper \cite{metafone-spina} by Metafune and Spina concerning elliptic operators generating analytic semi-groups in $L^p(\R^n)$, who focused on the following identity:
$$
\int_{\R^n}
u |u|^{p-2}\Delta u \dee x = - (p-1)
\int_{\R^n}
|u|^{p-2} \| \nabla u\|^2 \dee x ,
$$ 
where $u \in W^{2,p}(\R^n )$, $1 < p < \infty$. We refer to Section~\ref{SectionMetafune} for a more detailed treatment.

Another one is presented  in Section~\ref{chain-sect},
where we conclude the following integral chain-rule-type upper bound 
$$ \int_\Omega |P(\widetilde{H}(u))|\dee x \lesssim \int_\Omega |\widetilde{H}^{'}(u)  Pu|\dee x, $$ 
where $P$ is the elliptic operator as in \eqref{pe}. Note that the pointwise chain-rule $P(\widetilde{H}(u))= \widetilde{H}^{'}(u)  Pu$ obviously does not hold. The pointwise chain-rule-type formulae for the infinitesimal generators of the diffusion processes were derived in \cite[Lemma~1, p.~171]{bakry-emery}. A discussion of motivations coming from the probability and potential theories, including several open problems, will be presented in Section~\ref{probabilityka}. 

\bigskip

Let us mention that second-order operators in the non-divergent form play a crucial role in the theory of elliptic PDE’s (see e.g.~\cite{[GT]}). In the probality theory, these operators are an important subclass of generators of Feller semi-groups. For example, \cite[p.~71, Theorem~3.3]{[BSchW]} says that an elliptic second-order operator in non-divergent form, with very weak assumptions on the coefficients, generates a Feller semi-group. They also appear as generators of  It\^{o} diffusion processes, see e.g.~\cite[p.~123, Theorem~7.3.3]{[Oks]}. Moreover, non-divergent operators appear in the Kolmogorov equations backwards (\cite{[SV]}, \cite{[T]}).

\bigskip

By deriving the identities and inequalities, we intend to contribute to both analysis of PDEs and potential theory, where one deals with elliptic operators in non-divergent form.

\section{Preliminaries and notation}\label{preliminary}

\subsection{General notation}

We use the following notation:

\smallskip
\noindent
 {\bf vectors and matrices:} $\bullet$ by $v^T$ we denote the transposition of a given vector or matrix $v$ $\bullet$ when $a\in \R^n,b\in\R^n$, by $a\otimes b$ we mean the matrix $(a_ib_j)_{i,j\in\{ 1,\dots,n\}}\in \R^{n\times n}$;

\smallskip
\noindent
 {\bf divergence:} the divergence of the matrix field $\A=(\A^1,\dots , \A^n)\in \mathbf{R}^{n\times n}$ (where $A^i\in \mathbf{R}^n$) is the vector field $\operatorname{div}\A:= ({\rm div }\A^1,\dots , {\rm div }\A^n)$, where $\operatorname{div} \mathbf{A}^i$ is the standard divergence;
    
\smallskip
\noindent
 {\bf Hausdorff measure:} by $\sigma$ we denote the $n-1$ dimensional Hausdorff measure defined on $\mathbf{R}^n$, eventually restricted to its subsets; 

\smallskip
\noindent
{\bf function spaces:} $\bullet$ we use the standard notation for Sobolev spaces 
defined on $\Omega$: $W^{m,p}(\Omega,E)$, where $E$ is the target Euclidean space, omitting $E$ when $E=\mathbf{R}$
$\bullet$ by $W_{\rm loc}^{m,p}(\Omega,E)$ we denote their local counterparts
$\bullet$ by $W_0^{m,p}(\Omega ,E)$ we mean the completion of smooth compactly supported mappings in $W^{m,p}(\Omega ,E)$;

\smallskip
\noindent
 {\bf extension by zero:} when $v$ is a function defined on some subset $S$ of a Euclidean space, by $v\chi_S$ we mean the function $v$ extended by zero outside set $S$;
 
\smallskip
\noindent
 {\bf estimate notations:} by $a\lesssim b$ we mean that there exists some universal constant $C>0$ such that $a\le Cb$, its dependence can be specified; if $a\lesssim b$ and $b\lesssim a$ we denote $a\approx b$. In the key estimates, we inform about the precise constant.

\subsection{Basic assumptions}\label{basic-set-ass}

We shall now establish the common assumptions and conventions that will be used throughout the paper.

\subsubsection*{Assumptions about $\Omega$}

\begin{enumerate}[label={\bf ($\mathbf{\Omega}$)}] 
\item \label{(Omega)} We assume that $\Omega\subseteq \mathbf{R}^n$, where $n\ge 2$, is a bounded domain with Lipschitz boundary ($\Omega \in C^{0,1}$), see e.g. \cite{ma}.
By $n(x)$, we denote the outer normal to $\partial\Omega$, defined $\sigma$-almost everywhere on $\partial \Omega$.
\end{enumerate}

\subsubsection*{Assumptions about the elliptic operator}

We deal with the elliptic operator given in non-divergent form by
\begin{eqnarray}\label{pe}
Pu=\sum_{i,j}a_{i,j}(x)\frac{\partial^2 u}{\partial x_i\partial x_j}(x) \text{ a.e.~in $\Omega$, }  u\in W^{2,1}_{\rm loc}(\Omega), 
\end{eqnarray}
where we will consider diverse assumptions on the ellipticity  matrix \\$\A(x):=\{ a_{i,j}(x)\}_{i,j\in \{ 1,\dots ,n\}}\in~\mathbf{R}^{n\times n}$, as will be required by the various statements. The most basic assumption reads as follows:
\begin{enumerate}[label={\bf (A1)}]
    \item \label{(A1)} {\bf uniform ellipticity and simplest regularity:}  $\A(x)$ is symmetric and positive definite, the coefficients are continuously differentiable up to the boundary, i.e.~$a_{i,j}(\cdot) \in C^1\left ( \overline{\Omega} \right)$, that is they possess an extension to $C^1$ functions defined in some neighbourhood of $\overline{\Omega}$, satisfying  
    \begin{eqnarray*}
    \begin{aligned}
        &c_{\A}\|\xi\|^2 \le \xi^T\A(x)\xi \le C_{\A} \|\xi\|^2 ,\text{ for } c_{\A}, C_{\A}>0 \text{ independent of } x\in\overline{\Omega}\text{ and } \xi\in \mathbf{R}^n,\\
        &d_{\A} :=\| \operatorname{div} {\A}\|_{L^\infty (\Omega)}^2 c_{\A}^{-1}.
    \end{aligned}
    \end{eqnarray*}
\end{enumerate}

\noindent
{\it The corresponding norm.} For the fixed $x\in\overline{\Omega}$ we denote the norm  on  $\mathbf{R}^n$ corresponding to the matrix
$\A (x)$, induced by  scalar product $\langle z , y \rangle_{\A (x)} := z^T \A (x) y$ 
\begin{align}\label{eq: A norm definition}
\|y\|_{\A (x)}&:=\sqrt{y^T\A(x) y}, \text{ } y \in \mathbf{R}^n.
\end{align}
Let us stress that it follows from our assumption \ref{(A1)} that all the norms $\| \cdot \|_{\A(x)}$ are equivalent to the Euclidean one and this equivalence is uniform with respect to $x \in \overline{\Omega}$. We will usually omit $x$ in the notation to simplify the presentation. 

In Section~\ref{sectionSimplif} we will also use the following stronger condition for the operator $\A$.

\begin{enumerate}[label={\bf (A2)}]
\item \label{(A2)} {\bf stronger condition for
$\A(\cdot)$:} $a_{i,j}\in C^2\left ( \overline{\Omega} \right)$ and $\operatorname{div}^{(2)}\A:=\sum_{i,j\in \{1,\dots,n\} }\frac{\partial^2a_{i,j}}{\partial x_i\partial x_j}\le 0$ in $\Omega$.
\end{enumerate}

\subsubsection*{Principal weight and its transformations}
We will be using the following weights and their transformations, introduced first in \cite{akjp}.

\begin{enumerate}[label={\bf (h)}]
\item \label{(h)} {\bf principal weight:}
 When $0<B\le\infty$, by principal weight, we will mean the given continuous positive function  $h:(0, B)\to (0,\infty)$.
\end{enumerate}

\begin{itemize}
\item {\bf special transformations:}
Having given the principal weight $h$ and its first-order antiderivative $H$, that is $H'=h$,   we will deal with the following  transforms, defined for $s \in (0,B)$:
\begin{eqnarray}\label{gh}
\mathcal{T}_{H}(s):= \frac{H(s )}{h(s)},  \ \ \ 
\mathcal{G}_{H}(s):= \frac{H^2(s )}{h(s)} = \mathcal{T}_{H}^2(s)h(s). 
\end{eqnarray} 
\item Moreover, we will use the notation $\widetilde{H}$ for the second-order antiderivative of $h$, so that~$\widetilde{H}''=h$. The functions $H$ and $\widetilde{H}$ are not uniquely determined and will be specified if needed.
\end{itemize}

\begin{Remark}\rm
    Note that the positivity of $h$ implies that $H$ is strictly increasing and $\widetilde{H}$ is strictly convex.
\end{Remark}

\begin{Remark}[about $\mathcal{T}_{H}(s)$]\label{teha}\rm~
\begin{enumerate}
\item
When $H$ is not zero, then $\mathcal{T}_{H}=\tfrac{1}{(\ln |H(s)|)'}$, for $s\in (0,B)$.
\item \label{teha_ii} In many situations $\mathcal{T}_{H}(s)\approx s$ for $s\in (0, B)$.\\
Consider e.g.~$B=\infty$ and the following couples $(H,h)$:
\begin{eqnarray*}
 H(s)&: =&(\alpha+1)^{-1}
s^{\alpha+1},\  h(s) =s^\alpha, \text{ where } \alpha \neq-1;\\
H(s)&: =& s^{\alpha +1}\ln^\beta(e+s),\   h (s ) \approx s^\alpha\ln^\beta(e+s),\  \alpha>-1,\ \beta\neq 0.
\end{eqnarray*}
\item The following example shows that the transform $\mathcal{T}_{H}(s)$  can be an arbitrary positive continuous function.
Let $\tau (s )$ be any continuous positive function defined on $(0,B)$ and let $\beta (s )$ be any antiderivative of 
$\frac{1}{\tau (s)}$, that is $\beta'=\frac{1}{\tau}$.
Then the function $H(s ) := {\rm exp}(\beta (s))$ satisfies: 
$ h(s):= H'(s )= \frac{1}{\tau (s )}H(s),$ and so 
$\mathcal{T}_{H}(s) = \tau (s)$.
\end{enumerate}
\end{Remark}

\begin{Remark}[other forms of $H(s)$ and $\mathcal{G}_{H}(s)$]\rm
Observe that $\mathcal{G}_{H}(s)$ and $H(s)$ represent also as:
\begin{eqnarray}
\mathcal{G}_{H}(s)=   \mathcal{T}^2_{H}(s)h(s)= \mathcal{T}_{H}(s)H(s),   \ \ \ 
H(s)=  \mathcal{T}_{H}(s)h(s),  \  s \in (0,B).\label{secondtransform1}
\end{eqnarray}
\end{Remark}

\begin{Remark}[weight transforms]\rm 
In the special case when $h$ was integrable near $0$,
 for $C\in\mathbf{R}$, the following notation was used in \cite{akjp}, 
\begin{eqnarray*}
H_C(s)&:=&\int_0^s h(t)\:\dee t -C,\ \  s\in[0,B),\ \ \ 
\mathcal{T}_{h,C} :=\frac{H_C(s)}{h(s)}.
 \label{weighttransf}
\end{eqnarray*}
Note that  $H_C$  is of class $C^1$ on $(0, B)$ and that $H_C$ can be continuously extended to zero, which might not be the case of   $h$.
\end{Remark}

\subsection{Main assumptions}\label{assumptions}

\noindent
Firstly, we present what is the general assumption in most of our results, namely Theorems~\ref{LemmaBoundedU}, \ref{LemmaBoundedUcoinfty}, \ref{structureA}, and \ref{simplification2}:

\begin{enumerate}[label={\bf (G)}]
\item \label{(G)}$\Omega\subseteq \mathbf{R}^n$, $P$,  $\A(\cdot)$,  $h(\cdot)$ defined on $(0,B)$, where $0<B\le\infty$,
  as in: \ref{(Omega)}, $\eqref{pe}$, \ref{(A1)}, \ref{(h)},
$H$ is the given antiderivative of $h$.
\end{enumerate}
\noindent
The weakest version of our main result, namely Theorem~\ref{LemmaBoundedU}, is proved under the following assumption:

\begin{enumerate}[label={\bf (u)}]
\item \label{(u-H)} $u\in W_{\rm loc}^{2,1}(\Omega)$, $u(x)\in (0,B)$ almost everywhere in $\Omega$, for every antiderivative $\widetilde{H}$ of $H$ we have 
$\widetilde{H} (u)\in W^{2,1}(\Omega)$. Moreover, $Pu \; H(u)\in L^1(\Omega)$.
\end{enumerate}

\begin{Remark}[equivalent form of \ref{(u-H)}]\label{equiv-Pu}\rm
Let us note that the condition $Pu \; H(u)\in L^1(\Omega)$ in the assumption \ref{(u-H)} above could be equivalently substituted by the condition $h(u)\|\nabla u\|^2_{\A}\in L^1(\Omega)$, this is a consequence of the formula \eqref{pointwise-identity} from the proof of Theorem~\ref{LemmaBoundedU}.    
\end{Remark}

\begin{Remark}[the choice of second-order antiderivative in \ref{(u-H)} ]\label{o-widehat}\rm
Any two antiderivatives of $H$ defined on $(0, B)$ can only differ by a constant. Therefore, in the assumption \ref{(u-H)} in place of the expression ``for every antiderivative $\widetilde{H}$ of $H$'' we could write as well ``for some antiderivative $\widetilde{H}$ of $H$''.
\end{Remark}

In Section \ref{composition-sobolev}, precisely in Examples 
\ref{zlozenie} and \ref{homework1}, we show that  the assumption $\widetilde{H}(u)\in W^{2,1}(\Omega)$ implies neither $u\in W^{2,1}(\Omega)$ nor $u\in W^{2,1}_{\rm loc}(\Omega)$.

\bigskip

\noindent
The next assumption is used in Theorems~\ref{LemmaBoundedUcoinfty}, \ref{structureA}, \ref{Opial}, and \ref{simplification2} to allow for $u \equiv 0$ or $u\equiv B$ on sets of positive measure in $\Omega$:

\begin{enumerate}[label={\bf (I)}]
\item \label{(I-H)} $I$ is one of the intervals
$(0,B),[0,B)$, $(0,B]$ or $[0,B]$; the last two cases are allowed only when $B<\infty$. $H'=h$ on $(0,B)$ and some (so also any) antiderivative $\widetilde{H}$ of $H$    continuously extends from $(0,B)$ to $I$.  
We additionally assume that $I$ is maximal possible subinterval of $[0,B]$  such that $\widetilde{H}$ has a continuous extension on $I$.
Such an extension will be also denoted by $\widetilde{H}$ when there is no risk of confusion.
\end{enumerate}

An extended version of the next Remark, which contains detailed arguments for the claims, can be found in Appendix~\ref{section Appendix}, Remark~\ref{rem-ii}.

\begin{Remark}[about the condition \ref{(I-H)}]\label{rem-ii_abbrv} \rm
    Recall that $H$ is strictly increasing on $(0,B)$, and thus $\widetilde{H}$ is strictly convex. In particular, $\widetilde{H}$ is strictly monotone near the endpoints of $(0,B)$ and the limits $\lim_{s\to 0_+}\widetilde{H}(s)$ and $\lim_{s\to B_-}\widetilde{H}(s)$ exist, finite or not. If either of the limits is finite, then $\widetilde{H}$ can be continuously extended to the corresponding endpoint by setting the value to be the respective limit. Moreover, $H$ is integrable near either of the endpoints if and only if the respective limit is finite, in which case we may  for example construct an antiderivative of $H$ for $s \in I$ as either the Hardy or conjugate-Hardy transform of $H$.

     On the other hand, we never require that the functions $H$ and $h$ can be extended to $I$.
\end{Remark}

\noindent
The condition below appears in Theorems:~\ref{LemmaBoundedUcoinfty}, \ref{structureA}, and \ref{simplification2}: 

\begin{enumerate}[label={\bf (u-I)}]
\item  \label{(u-I-H)}  $u\in W_{\rm loc}^{2,1}(\Omega)$, $u(x)\in I$ almost everywhere in $\Omega$, where
$I$ is as in \ref{(I-H)} and for $\widetilde{H}$ satisfying $\widetilde{H}'=H$ on $(0,B)$, it holds $\widetilde{H}(u)\in W^{2,1}(\Omega)$.  Moreover, $Pu \; H(u)\chi_{u\in(0,B)}\in L^1(\Omega)$. 
\end{enumerate}

\begin{Remark}[about the condition \ref{(u-I-H)} ]\rm
By the same arguments as in Remark \ref{equiv-Pu} we  note  that the condition  $Pu \; H(u)\chi_{u\in (0,B)}\in L^1(\Omega)$, in the assumption \ref{(u-I-H)} above, could be equivalently substituted by the condition $h(u)\|\nabla u\|^2_{\A}\chi_{u\in (0,B)}\in L^1(\Omega)$.    
\end{Remark}

\noindent
Theorem~\ref{structureA} also requires the following Dirichlet-type condition, which will be analysed in Section~\ref{u-tildeh}, Theorem~\ref{RemAn(u-TildeH-1)}:
\begin{enumerate}[label={\bf (u-$\widetilde{\text{H}}$)}]
\item \label{(u-TildeH)} The assumptions \ref{(I-H)} and  \ref{(u-I-H)} hold and 
$H$  
possesses the  antiderivative $\widetilde{H}$ defined on $I$ such that
\begin{equation*}
\widetilde{H}\ge 0\text{ on }  (0,B), \text{ } \widetilde{H}(u) \in W^{2,1}(\Omega), \text{ and }    \widetilde{H}(u)\equiv 0 \text{ on }  \partial\Omega.
\end{equation*}
\end{enumerate}

\noindent

Finally, in Theorems~\ref{Opial} and \ref{simplification2} we exploit the following property of $h$ and its given first and second-order antiderivatives $H$ and $\widetilde{H}$, respectively, which seems to be of independent interest:

\begin{enumerate}[label={\bf ($\mathcal{G}_H$)}]
\item \label{(h-H-TildeH)} There exists a constant $C_{\widetilde{H}}>0$ such that the following condition holds: 
\begin{eqnarray*}
\mathcal{G}_H(s)=  \frac{H^2(s)}{h(s)}\le C_{\widetilde{H}}|\widetilde{H}(s)|,
\end{eqnarray*}		   
where $\mathcal{G}_H$ is defined in \eqref{gh}.
 \end{enumerate}

\begin{Remark}[about the condition \ref{(h-H-TildeH)} ]
\rm This condition is equivalent to the estimate
\begin{equation}\label{punktmulti}
|\widetilde{H}^{'}(s)|^2 \le C_{\widetilde{H}} |\widetilde{H}(s)||\widetilde{H}^{''}(s)|.
\end{equation}
An interpretation of this estimate is that the first-order derivative of $\widetilde{H}$, i.e.~$H$, is controlled from above by the geometric mean of $\widetilde{H}$ and $h = \widetilde{H}^{''}$.

Note that \eqref{punktmulti} can be interpreted as a pointwise Gagliardo--Nirenberg-type inequality. However, it does not hold in general, but some its modified variants can be found in literature. Usually, they involve the maximal operator or other averaging operators, see e.g.~\cite{agnieszka-pointwise, LRS, mazya-shap, MazyaKufner86}.

The class of functions satisfying \eqref{punktmulti} includes the power-, power-logarithmic-, and exponential-type functions, e.g.~$\widetilde{H}(s)=s^\alpha$, $\widetilde{H}(s)=s^\alpha(\ln(2+s))^\beta$, $\widetilde{H}(s)=e^{\beta s^\alpha}$, for proper choices of the parameters involved. 

\end{Remark}

\subsection{Properties of Sobolev functions}\label{SectionPropSobolevFunc}

\subsubsection*{The pointwise value of Sobolev function}
When $w\in W^{1,1}(\Omega)$, $\Omega$ being as in \ref{(Omega)}, then its value can be defined at every point $x\in \overline{\Omega}$ by  choosing the canonical Borel representative of $w$, given by the formula
\begin{equation}\label{valueatpoint}
w(x):=\limsup_{r\to 0_+} \frac{1}{|\Omega\cap B(x,r)|}\int_{\Omega\cap B(x,r)} w(y)\dee y.
\end{equation}
It is known that such a representative coincides with any other one ${\mathcal L}^n$-almost everywhere in $\Omega$ and it coincides $\sigma$-almost everywhere on $\partial\Omega$ with the trace of $w$, see e.g. \cite{anzoletti-giaquinta}. In particular, this definition also applies when the value of $w$ is prescribed only inside $\Omega$, and can be used to define the trace of $w$ on $\partial\Omega$.

For an in-depth treatment of the trace theory, we refer the reader to e.g.~\cite[Section~18]{Leoni17} or \cite[Section~6]{KufnerJohn77}. However, we will only need the most classical result, which is the existence of the trace operator for Lipschitz domains (e.g.~\cite[Theorem~18.1]{Leoni17} or \cite[Theorem~6.4.1]{KufnerJohn77}). Further comments are presented in the proof of Lemma \ref{lem-app1} in Section \ref{section Appendix}.

\subsubsection*{Poincar\'e inequality}

We will use  the following variant of  Poincar\'e inequality (see \cite[Theorem~13.19]{Leoni17}):
\begin{equation}\label{poincare}
\int_\Omega |w(x)|\dee x \le C_P\int_\Omega \|\nabla w (x)\|\dee x,\text{ where } w\in W_0^{1,1}(\Omega),
\end{equation}
which holds for $\Omega$ satisfying \ref{(Omega)} and where $C_P>0$ does not depend on $w$.

In some cases, we will also need the more general version that assumes only $w \in W^{1,1}(\Omega)$ (instead of $w\in W_0^{1,1}(\Omega)$):
\begin{equation*}
\int_\Omega |w(x) - w_{E}|\dee x \le C \int_\Omega \|\nabla w (x)\|\dee x,
\end{equation*}
where $C$ does not depend on $w$, $E \subseteq \Omega$ is an arbitrary measurable set of finite non-zero measure, and
\begin{equation*}
    w_{E} = \frac{1}{| E |} \int_{E} w(x) \dee x,
\end{equation*}
which holds for $\Omega$ satisfying \ref{(Omega)}. See e.g.  \cite[Theorem~13.27]{Leoni17} for details.

\begin{Remark}[Equivalence of homogeneous and inhomogeneous Sobolev spaces on regular domains] \label{Poincare_equiv}\rm
   We recall that it is a consequence of the general Poincar{\' e} inequality that, on sufficiently regular bounded domains (e.g.~those satisfying \ref{(Omega)}), any locally integrable function $w$, whose distributional gradient satisfies  $\nabla w \in L^p(\Omega, \R^n)$, belongs itself to $L^p(\Omega)$, where $p \in [1, \infty)$. See e.g.~\cite[Section 1.1.11]{ma}.
\end{Remark}

\subsubsection*{ACL characterisation property}

 We will need the following variant of Nikodym ACL
Characterisation Theorem, which can be found
e.g. in \cite[Section 1.1.3]{ma}. Let us note that it holds in the Euclidean space but can be applied appropriately inside the regular domain as well.

\begin{Theorem}[Nikodym ACL Characterisation Theorem]\label{nikos}\phantom{a}
\begin{enumerate}
\item Let  $w\in~W^{1,1}_{\rm loc}(\rn)$. Then for every $i\in \{
1,\dots ,n\}$ and for almost every $a\in \R^{i-1}\times \{
0\}\times \R^{n-i}$ the function
\begin{equation}\label{tratt}
\R\ni t\mapsto w(a+te_i)
\end{equation}
is locally absolutely continuous on $\R$. In particular, for almost
every point $x\in\R^n$ the distributional derivative $\frac{\partial w}{\partial x_i}$ is
the same as the classical derivative at  $x$.
\item
Assume that $w\in L^1_{\rm loc}(\R^n)$ and
for every $i\in \{ 1,\dots,n\}$ and for almost every $a\in \R^{i-1}\times \{
0\}\times \R^{n-i}$ the function in \eqref{tratt} is
locally absolutely continuous on $\R$ and all the derivatives
$\frac{\partial w}{\partial x_i}$ computed almost everywhere
are locally integrable on $\R^n$. Then $w$ belongs to
$W^{1,1}_{\rm loc}(\R^n)$.
\item
Let $1\le p\le \infty$,   $w\in W^{1,1}_{\rm loc}(\rn)$ and $\Omega\subset \R^n$ be an open
subset. Then $w$ belongs to $W^{1,p}(\Omega)$ if and only if $w\in L^p(\Omega)$ and
every derivative $\frac{\partial w}{\partial x_i}$ computed almost everywhere belongs to the space $L^{p}(\Omega)$.
\end{enumerate}
\end{Theorem}

\noindent

The following easy consequence of the ACL characterisation is left to the reader.

\begin{Lemma}\label{lcomp}
If $f: [-R,R]\rightarrow \R$ is absolutely continuous with values in the interval $[\alpha,\beta]$ and $L: [\alpha,\beta] \rightarrow\R$ is a  Lipschitz function, then
the function $(L\circ f)(x):=L(f(x))$ is absolutely continuous on $[-R,R]$. In particular, if 
$L$ is locally Lipschitz and $w$ belongs to $W^{1,1}_{\rm loc}(\Omega)$, where $\Omega\subseteq \Rn$ is an open domain,  then $L\circ w$ belongs to $W^{1,1}_{\rm loc}(\Omega)$.
\end{Lemma}
\section{Derivation of the identity and inequalities} \label{SectionBasicResults}

\subsection{First simplest approach} \label{SectionFirsResult}

At first, we deal with the simplest assumptions about the principal non-linearity $h$; that is, we do not care if it is defined at the endpoints of the interval $(0, B)$. On the other hand, the admitted function $u$ must satisfy the assumption $u\in(0, B)$ a.e.~in $\Omega$. 

\smallskip
\noindent
Our first statement reads as follows.

\begin{Theorem}[the identity and the inequalities]\label{LemmaBoundedU}
Let the assumptions \ref{(G)} and \ref{(u-H)} be satisfied. Then the following properties hold:   
\begin{enumerate}
\item \label{Theorem_identity_inequality_i} {\bf the identity:}
\begin{equation} \label{identity}
\begin{split}
         \int_{\Omega } \|\nabla u\|_{\A}^{2} h(u(x))\dee x  =& -\int_{\Omega} Pu \; H(u) \dee x   \\
     & -\int_{\Omega } \operatorname{div} \A\cdot \nabla u \; H (u) \dee x +\Theta, \text{ where} \\
\Theta =& \int_{\partial\Omega}  n(x)^T \A(x) \nabla  (\widetilde{H} (u))\dee \sigma (x). 
\end{split}
\end{equation}
Moreover, all the involved integrals converge. 

\item \label{Theorem_identity_inequality_ii} {\bf the inequalities:} 
\begin{eqnarray}\label{ineqal-simple}
   \int_{\Omega } \|\nabla u\|_{\A}^{2} h(u)\dee x  
&\le&   \int_{\Omega } |Pu||H(u)| \dee x + \Theta \text{ when } \operatorname{div}{\A}\equiv 0, \\
   \int_{\Omega } \|\nabla u\|_{\A}^{2} h(u)\dee x  
&\le&  d_{\A} \int_{\Omega }\mathcal{G}_{H}(u)\dee x + 2\int_{\Omega } |Pu||H(u)| \dee x + 2\Theta \text{ otherwise,}\nonumber  \\\label{ineqality}
\end{eqnarray}
where $d_{\A}$ is as in \ref{(A1)}.
\end{enumerate}
\end{Theorem}

\noindent
Before we present the proof, we start with the following remarks.

\begin{Remark}[computation of boundary term]\label{tetaintegrability}\rm
The value of $\nabla(\widetilde{H}(u))$ at $\partial\Omega$
 in \eqref{identity} is computed by the formula \eqref{valueatpoint}. It is well defined and integrable over $\partial\Omega$ as $\nabla(\widetilde{H}(u)) \in W^{1,1}(\Omega,\R^n)$, and we know that $\nabla (\widetilde{H}(u))= \nabla u \;  H(u)$ in $\Omega$. Note that our statement does not depend on the choice of $\widetilde{H}$.
 \end{Remark}
 
\begin{Remark}[importance of the term $\int_{\Omega } \mathcal{G}_{H}(u)\dee x$]\label{G-finite}\rm
Note that the term $d_{\A}\int_{\Omega } \mathcal{G}_{H}(u)\dee x$ is zero if the operator $P$ has constant coefficients or, more generally, if $\operatorname{div}\A \equiv 0$. We are not claiming that the term $\int_{\Omega } \mathcal{G}_{H}(u)\dee x$ is finite.
We will discuss later, in Section~\ref{sectionSimplif}, the situations when the term
$\int_{\Omega } \mathcal{G}_{H}(u)\dee x$ is finite and can be omitted from the inequality \eqref{ineqality}.
\end{Remark}

\begin{Remark}[interpretation of  the term $\int_{\Omega } |Pu||H(u)| \dee x$]\rm
The term $\int_{\Omega } |Pu||H(u)|\dee x$ can be rewritten as
$$
\int_{\Omega } |Pu||H(u)|\dee x= \int_{\Omega } |Pu||\mathcal{T}_{H}(u)| h(u)\dee x 
= \int_{\Omega } \left( \sqrt{|Pu||\mathcal{T}_{H}(u)|}\right)^2 h(u)\dee x. 
$$
Assuming that the function $H$ is non-decreasing and satisfies
\begin{equation}
    |\mathcal{T}_{H}(s)|\lesssim s, \ \ \ {\rm equivalently}\ \frac{|H(s)|}{s} \lesssim H^{'}(s),\   {\rm where}\  s\in (0,B) \label{T<u},
    \end{equation} we obtain from \eqref{ineqality} the estimate
$$
\int_{\Omega } \|\nabla u\|_\A^2 h(u)\dee x \lesssim \int_{\Omega } |u|^2 h(u)\dee x +
 \int_{\Omega } \left( \sqrt{|Pu||u|}\right)^2 h(u)\dee x+\Theta,
$$
where the geometric mean value on the right-hand side of the inequality, i.e.~$\sqrt{|Pu||u|}$, replaces $\|\nabla u\|_{\A}$ on the left-hand side of the inequality. The condition \eqref{T<u} in the stronger version, that is $|\mathcal{T}_{H}(s)|\approx s$, has already been discussed in Remark~\ref{teha}, part~\ref{teha_ii}, where we presented examples among power and power-logarithmic functions. In addition to these examples, let us add that any  non-negative convex function $H$ such that $H(0)=0$ also satisfies \eqref{T<u}.
\end{Remark}

\begin{proof}[Proof of Theorem~\ref{LemmaBoundedU}]\ref{Theorem_identity_inequality_i}
The proof follows by steps.

\smallskip
\noindent
{\sc Step 1.} We first derive the pointwise identity
\begin{eqnarray}\label{pointwise-identity} 
L^1(\Omega)\ni  P (\widetilde{H}(u)) = h(u)\|\nabla u\|_{\A}^2 + H(u)Pu
=: \alpha(u) + \beta(u) \text{ a.e. in } \Omega ,
\end{eqnarray}
then we show that all terms involved in \eqref{identity} are finite.

By our assumption $\widetilde{H}(u)\in W^{2,1}(\Omega)$. Using the ACL characterisation property (Theorem~\ref{nikos} and Lemma~\ref{lcomp}),  we deduce that 
$\widetilde{H}(u)$ is absolutely continuous on almost all lines in $\Omega$ 
parallel to the coordinate axes and its distributional derivative can be computed almost everywhere. The same argument applies to

\begin{equation} \label{Def_v}
    v:=\nabla \widetilde{H}(u)=H(u)\nabla u\in W^{1,1}(\Omega,\R^n).
\end{equation}

Therefore, for almost every $x\in \Omega$, we have
\begin{equation}\label{tozsamosc1}
\nabla \widetilde{H}(u) = H(u)\nabla u, \ \ \
\nabla^{(2)} \widetilde{H}(u)= h(u)\nabla u\otimes \nabla u + H(u)\nabla^{(2)}u.
\end{equation}
This implies \eqref{pointwise-identity},  after we scalarly multiply 
second identity in \eqref{tozsamosc1} by the matrix $\A(\cdot)$, using the standard scalar product of matrices $\langle A,B\rangle =\sum_{i,j}a_{i,j}b_{i,j}$. 

By our assumptions, the terms $P(\widetilde{H}(u))$ and $\beta(u)$ in \eqref{pointwise-identity} are both integrable over $\Omega$. Therefore also $\alpha(u)$  is integrable over $\Omega$. Finally, it follows from \eqref{Def_v} and the classical Trace Theorem that $v \in L^1(\partial\Omega,\R^n)$, and so
$\Theta$ is finite (see Remark \ref{tetaintegrability}). Hence, all terms involved in \eqref{identity} are finite.

\smallskip
\noindent
{\sc Step 2.} We prove the integral identity \eqref{identity}.\\
As $v\in W^{1,1}(\Omega, \mathbf{R}^n)$ (see \eqref{Def_v}), so does $\A v= \A\nabla (\widetilde{H}(u)) = H(u)\A\nabla u$, because $\A\in C^1(\overline{\Omega},\mathbf{R}^{n\times n})$. We can compute its divergence, getting
\begin{equation}
\R \ni \int_{\Omega} \operatorname{div}(\A v) \dee x  =
\int_{\Omega} \operatorname{div}\left( \A \nabla \left ( \widetilde{H}(u) \right )\right) \dee x 
= \int_{\partial\Omega}  n(x)^T \A(x)  \nabla  (\widetilde{H}(u) )\dee \sigma (x)=\Theta . \label{brzegowy} 
\end{equation}

Notice that 
$$
\operatorname{div}(\A \nabla u)= \A\cdot \nabla^{(2)}u + \operatorname{div}\A\cdot \nabla u  = Pu + \operatorname{div}\A\cdot \nabla u\ \hbox{\rm~~a.e.~in}\ \Omega .
$$
It is easy to see that $\A\nabla u\in W^{1,1}_{\rm loc}(\Omega,\R^n)$, while $H(u)\in W^{1,1}_{\rm loc}(\Omega)$ by Lemma~\ref{lcomp}. Hence, we may apply the ACL characterisation (see Theorem~\ref{nikos}) to the terms involved in the equation above and compute that for almost every $x\in\Omega$
\begin{eqnarray*}
\operatorname{div}\left( \A \nabla \widetilde{H}(u)\right) &=&
\operatorname{div}\left( (\A \nabla u)H(u)\right) \nonumber\\
&=& \operatorname{div} (\A \nabla u) H(u)
+ (\A \nabla u) \cdot \nabla H(u)\nonumber\\
&=& (\operatorname{div}\A\cdot\nabla u) H(u) + Pu H(u) + \|\nabla u\|_{\A}^2 h(u).
\end{eqnarray*}
When we integrate it over $\Omega$ and apply \eqref{brzegowy} we derive  \eqref{identity}, and so the statement \ref{Theorem_identity_inequality_i} is proven.

\bigskip
\noindent
\ref{Theorem_identity_inequality_ii} The inequality \eqref{ineqal-simple} follows from \eqref{identity}. Assume then that $d_{\A}>0$, that is 
$\operatorname{div}\A\not\equiv 0$.  We can additionally assume that $\int_{\Omega}
 \mathcal{G}_{H}(u) \dee x$ is finite, as otherwise the inequality obviously holds.
Let us denote
\begin{align*}
&\mathcal{I}:=\left( \int_{\Omega } \|\nabla u\|_{\A}^{2} h(u(x))\dee x \right) ^{\frac{1}{2}}, &\mathcal{J} :=  \int_{\Omega } |Pu||H(u)| \dee x,  &&\mathcal{K} := \left( \int_{\Omega }
 \mathcal{G}_{H}(u)\dee x\right)^{\frac{1}{2}}.
\end{align*}
By the Cauchy--Schwarz inequality applied with respect to the measure $\mu = h(u)\dee x$, using \eqref{secondtransform1}, we get
\begin{align}
\left \lvert \int_{\Omega} \operatorname{div} \A\cdot \nabla u \; H(u) \dee x \right \rvert &\le \| \operatorname{div} \A\|_{L^\infty (\Omega)} 
 \int_{\Omega } \|\nabla u\| |\mathcal{T}_{H}(u)|  h (u) \dee x\label{a}\\
&\le  \| \operatorname{div} \A\|_{L^\infty (\Omega)} \left( \int_{\Omega } \|\nabla u\|^2h (u) \dee x\right)^{\frac{1}{2}} \left( \int_{\Omega}
 \mathcal{G}_{H}(u) \dee x\right)^{\frac{1}{2}}
\nonumber \\
&\stackrel{\ref{(A1)}}{\le} \| \operatorname{div} \A\|_{L^\infty (\Omega)}  c_{\A}^{-\frac{1}{2}}\left( \int_{\Omega} \|\nabla u\|_{\A}^2h (u) \dee x\right)^{\frac{1}{2}} \left( \int_{\Omega }
 \mathcal{G}_{H}(u) \dee x\right)^{\frac{1}{2}} \nonumber \\
&= d_{\A}^{\frac{1}{2}} \; \mathcal{I} \; \mathcal{K}.
\nonumber 
  \end{align}

Let us denote
$$
c:= \mathcal{J}+\Theta, \ \ b:= d_{\A}^{\frac{1}{2}}  \mathcal{K} =  \| \operatorname{div} \A\|_{L^\infty (\Omega)}  c_{\A}^{-\frac{1}{2}}\left( \int_{\Omega }
 \mathcal{G}_{H}(u)   \dee x\right)^{\frac{1}{2}}  .
$$
The identity \eqref{identity} and the estimate \eqref{a} imply:
\begin{eqnarray}\label{rownanieI}
\mathcal{I}^2 \le c + \mathcal{I}b.
\end{eqnarray}

Using Young's inequality  $\mathcal{I}b\le \frac{1}{2}\mathcal{I}^2 +
\frac{1}{2}b^2$, and recalling the finiteness of $\mathcal{I}$ (as  deduced in Step 1), we obtain the required inequality
$$
\mathcal{I}^2\le 2c+ b^2.
$$
\end{proof}

\begin{Remark}[estimates for $\Theta$]\rm
The validity of  \eqref{rownanieI} for finite $\mathcal{I}$ implies that the terms $b$ and $c$ satisfy the estimate
\begin{equation*}
	\frac{b^2}{4} + c \geq 0.
\end{equation*}
Note that this estimate is non-trivial, as $c$ might be negative.
In particular, under the assumptions of Theorem~\ref{LemmaBoundedU}, the following trace-type inequality involving composition operator holds:
\begin{equation*}
-\int_{\partial\Omega}  n(x)^T \A(x) \nabla   (\widetilde{H}(u(x)))\dee \sigma (x)\le \frac{d_{\A}}{4} \int_{\Omega}
 \mathcal{G}_{H}(u(x)) \dee x + 
\int_{\Omega} |Pu(x)||H(u(x))| \dee x.
\end{equation*}
It seems that such an inequality might be of independent interest.
\end{Remark}

\subsection{Generalisation which admits $u \equiv 0$ or $u\equiv B$ on sets of positive measure in $\Omega$} \label{SectionSecondResult}

In our previous statement, we have assumed that $u\in(0, B)$ a.e.~inside $\Omega$. Such an assumption looks rather strong, as it excludes, for example, smooth functions compactly supported inside $\Omega$. 
To include this situation, we apply some extra assumptions about the principal non-linearity $h$. So far $h$ has been defined on $(0, B)$ only, so we have not been able to consider $u$ having the values either $0$ or $B$ on the set of positive measure inside $\Omega$. As we show in this section, such functions can be admitted to the inequalities when we know that the second antiderivative of $h$, that is $\widetilde{H}$, which is a priori defined on $(0, B)$, can be continuously extended to some of the endpoints of $(0, B)$. If such an endpoint is $0$, the generalisation includes smooth compactly supported functions, under some extra regularity assumptions. 

In our next result we will require the additional assumptions \ref{(I-H)} and \ref{(u-I-H)} as formulated in Section~\ref{assumptions}. We also recall Remark~\ref{rem-ii_abbrv} which provides context for the assumption \ref{(I-H)}.

We have the following statement, which generalises Theorem~\ref{LemmaBoundedU}.

\begin{Theorem}[generalization admitting $u\in \{ 0,B\}$]\label{LemmaBoundedUcoinfty}
Assume that \ref{(G)}, \ref{(I-H)}, and \ref{(u-I-H)} hold. Then we have:
\begin{enumerate}
\item {\bf the identity:}
\begin{equation} \label{identity1} 
\begin{aligned} 
    \int_{\Omega \cap \{ u\in(0,B)\} } \|\nabla u\|_{\A}^{2} h(u(x))\dee x  =& -\int_{\Omega\cap \{ u\in(0,B)\}  } Pu \; H(u) \dee x  \\
    &-\int_{\Omega \cap \{ u\in(0,B)\} } \operatorname{div} \A\cdot \nabla u \;  H (u) \dee x +\Theta, \text{ where} \\
\Theta :=& \int_{\partial\Omega }  n(x)^T \A(x) \nabla ( \widetilde{H} (u))\dee \sigma (x).    
\end{aligned}
\end{equation}

Moreover, all the involved integrands are integrable over their domains.

\item {\bf the inequalities:}
\begin{equation}\label{ineqality1}\begin{aligned}
  \int_{\Omega \cap \{ u\in(0,B)\} } \|\nabla u\|_{\A}^{2} h(u)\dee x  
\le&   \int_{\Omega \cap \{ u\in(0,B)\}   } |Pu||H(u)|\dee x + \Theta, \text{ when } \operatorname{div}{\A}\equiv 0, \\
   \int_{\Omega \cap \{ u\in(0,B)\} } \|\nabla u\|_{\A}^{2} h(u)\dee x  
  \le& 2\int_{\Omega \cap \{ u\in(0,B)\}   } |Pu||H(u)|\dee x + 2\Theta\\&+ d_{\A} \int_{\Omega  \cap \{ u\in(0,B)\} }\mathcal{G}_{H}(u)\dee x \text{ otherwise}, 
\end{aligned}\end{equation}
where $d_{\A}$ is as in \ref{(A1)}.
\end{enumerate}
\end{Theorem}

\noindent

\begin{proof}[Proof of Theorem~\ref{LemmaBoundedUcoinfty}] 
The case of $I=(0, B)$ was proven in Theorem~\ref{LemmaBoundedU}.
The proof of the more general statement is similar to the proof of Theorem~\ref{LemmaBoundedU}. We only sketch it, pointing out the following modifications:
\begin{itemize}
\item The condition $\widetilde{H}(u)\in W^{2,1}(\Omega)$ implies that for almost every $x\in\Omega$
\[
\nabla \widetilde{H}(u) (x)= \left\{
\begin{array}{ccc}
H(u(x))\nabla u (x) & \text{ if } & u(x)\in (0,B)\\
0  & \text{ if } & u(x)\in \{ 0,B\}
\end{array}
\right.
= H(u(x))\nabla u (x)\chi_{u(x)\in (0,B)} .
\]
This is because of the chain rule and the fact that sets:
$\{ x: u(x)=0\}$, $\{ x: u(x)=B\}$ are subsets of level sets of $\widetilde{H}(u)$. 
Observe that $ H(u(x))\nabla u (x)\chi_{u(x)\in (0,B)}\in W^{1,1}(\Omega,\R^n)$, so we can define its trace by using the formula \eqref{valueatpoint}.

\item As $ H(u(x))\nabla u (x)\chi_{u(x)\in (0, B)}\in W^{1,1}(\Omega,\R^n)$, we deduce from the ACL characterisation property (see Theorem~\ref{nikos} and Lemma~\ref{lcomp}) that for almost every $x \in \Omega$
\begin{equation*}
	\nabla^{(2)} \widetilde{H}(u)= (h(u)\nabla u\otimes \nabla u + H(u)\nabla^{(2)}u) \; \chi_{u(x)\in (0,B)}.
\end{equation*}

\item In place of \eqref{pointwise-identity}
 we now have
 \begin{eqnarray}
L^1(\Omega)\ni  P (\widetilde{H}(u)) = h(u)\|\nabla u\|_{\A}^2 \; \chi_{u\in (0,B)} + H(u)Pu \; \chi_{u\in (0,B)}\nonumber\\
=: \alpha(u) + \beta(u),  \text{ a.e.~in } \Omega .\nonumber
\end{eqnarray}
\item Consequently, all the integrals that were originally considered over $\Omega$ are now restricted to 
$\Omega\cap \{ u\in (0,B)\}$.
\end{itemize}
\end{proof}

\subsection{Analysis of the condition $\widetilde{H}(u)\in W^{2,1}(\Omega)$}\label{composition-sobolev}

In the following two examples, we show that the condition

$\widetilde{H}(u)\in W^{2,1}(\Omega)$ does not imply $u\in W^{2,1}(\Omega)$, nor
$u\in W^{2,1}_{\rm loc}(\Omega)$. In Example \ref{zlozenie} we have $\widetilde{H}(u)\in W^{2,1}_0(\Omega)$ while $u\to \infty$ as $x\to \partial\Omega$.
Example \ref{homework1} shows that it might happen that $u\not\in W^{2,1}_{\rm loc}(\Omega)$, while 
$\widetilde{H}(u)\in W^{2,1}(\Omega)$.

\begin{ex}[$u\not\in W^{2,1}(\Omega)$, $\widetilde{H}(u)\in W^{2,1}_0(\Omega)$]\label{zlozenie}\em 
Let us consider $\Omega = B(0,1)\subseteq \R^n$, $n\ge 2$ and 
$$
u(x):= (1-|x|)^\alpha,\text{ where } \alpha <0,\ \ h(s)=s^\beta,\ \ \beta<0, \beta\not\in \{ -1,-2\}.$$
Such function does not belong to $W^{2,1}(\Omega)$. The argument is that the trace defined via \eqref{valueatpoint} is infinite, which is impossible for functions in $W^{2,1}$. However, $u$ is strictly positive and it belongs to $W^{2,1}_{\rm loc}(\Omega)$. On the other hand, when we take
$\widetilde{H}(s):= \frac{1}{(\beta +1)(\beta+2)}s^{\beta+2}$,  we have for $x\neq 0$
\begin{eqnarray*}
\widetilde{H}(u(x))= \frac{1}{(\beta +1)(\beta+2)}(|1-|x|)^{\alpha(\beta+2)},\, \nabla \widetilde{H}(u(x))
= \frac{-\alpha}{(\beta +1)}(1-|x|)^{\alpha(\beta+2)-1}
\frac{x}{| x|}.
\end{eqnarray*}
Further, denoting $\gamma_1:= \frac{\alpha(\alpha (\beta+2)-1)}{\beta +1}$ and $\gamma_2 := \frac{-\alpha}{\beta +1}$, we obtain
\begin{eqnarray*}
\frac{\partial^2}{\partial x_i\partial x_j}\widetilde{H}(u(x))
= \gamma_1 (1-|x|)^{\alpha (\beta+2)-2}\frac{x_ix_j}{|x|^2} + \gamma_2(1-|x|)^{\alpha (\beta+2)-1}\left(\frac{\delta_{ij}|x|^2 - x_ix_j}{|x|^3}    \right)\\
=:v_1+v_2,\text{ where } v_1\approx (1-|x|)^{\alpha (\beta+2)-2},\ \ v_2\approx (1-|x|)^{\alpha (\beta+2)-1}\frac{1}{|x|}.
\end{eqnarray*}

It is integrable over $\Omega$ (hint: use radial coordinates, the dimension does not matter)
provided that 
\begin{equation}\label{beta}
\alpha (\beta+2)-2>-1\Longleftrightarrow \beta <\frac{1}{\alpha}-2.
\end{equation}

In that case $\widetilde{H}(u)\in W^{2,1}(\Omega)$ and we easily verify that $\widetilde{H}(u)\in W^{2,1}_0(\Omega)$. However,  $u\not\in W^{2,1}(\Omega)$. Note also that $u\not\in W^{1,1}(\Omega)$, so that $\nabla u$ 
cannot be defined at $\partial\Omega$ as the trace of $W^{1,1}$ function. On the other hand,  $\nabla u H(u)$ belongs to
$W^{1,1}(\Omega)$, so its trace is defined at $\partial\Omega$.
In our case, the function $h(\cdot)$ does not continuously extend to zero, but $\widetilde{H}(u)$ does so when $\beta$ is as in \eqref{beta}.
\\
Note also that when $\alpha <-1$, then both $u(x) \to \infty$ and $|\nabla u (x)| \to \infty$ as $x \to \partial \Omega$.
\end{ex}

\begin{ex}[$u\not\in W^{2,1}_{\rm loc}(\Omega)$, $\widetilde{H}(u)\in W^{2,1}(\Omega)$]\label{homework1} \rm
In one dimensional case we can consider for example $\Omega= (-1,1)$, $u(x)={\rm sgn} (x)|x|^{\frac{1}{2}+\varepsilon}+1\in (0,2)$, where $0<\varepsilon<\frac{1}{2}$,  and $\widetilde{H}(s)=(s-1)^2$. In the case of general  dimension $n$, we can consider
$\Omega = (-1,1)^n$, 
$u(x_1,\dots,x_n)={\rm sgn} (x_1)|x_1|^{\frac{1}{2}+\varepsilon}+1$, with the same $\varepsilon$ and $\widetilde{H}$.
\end{ex}

\section{Simplifications of  \eqref{ineqality} and \eqref{ineqality1}  when $\operatorname{div}\A\not\equiv 0$}\label{sectionSimplif}

We shall now discuss possible simplifications of the inequalities \eqref{ineqality} and \eqref{ineqality1}. By simplification, we mean  
erasing the term $d_{\A}\int_{\Omega\cap \{ u\in(0, B)\} } \mathcal{G}_H(u)\dee x$, under some additional assumptions and at the cost of possible enlargement of the constants in the inequalities. Notice that when ${\rm div }\A\equiv 0$, such simplifications are unnecessary as this term is already zero. We propose two methods, as discussed below. 

\subsection{Simplifications based on sign conditions for $\operatorname{div}^{(2)}\A$ and  $\widetilde{H}(u)$ and on the Dirichlet condition $\widetilde{H}(u)\equiv 0$ on $\partial \Omega$}\label{simpl-sign}

We start from the following result, which mainly relies on the additional assumptions about the ellipticity matrix $\A$ and on the Dirichlet boundary condition and positive sign property for the composition $\widetilde{H}(u)$. We have the following result, which allows us to erase the terms $\Theta$ and 
$\int_{\Omega\cap \{ u\in(0,B)\}}\mathcal{G}_H(u)\dee x$ from the inequality \eqref{ineqality1}.

\smallskip
\noindent

For this, we will use the  conditions \ref{(A2)} and \ref{(u-TildeH)} introduced in Sections \ref{basic-set-ass} and \ref{assumptions}.

\begin{Theorem}[inequalities based on Dirichlet and sign conditions]\label{structureA}
Let the assumptions \ref{(G)}, \ref{(I-H)}, \ref{(u-I-H)}, and \ref{(u-TildeH)} hold, then:
\begin{enumerate}
\item \label{Theorem_Dirichlet_Sign_ii} We have

$$
\Theta =\int_{\partial\Omega }  n(x)^T \A(x) \nabla  \widetilde{H}(u(x))\dee \sigma (x)\le 0.
$$

\item \label{Theorem_Dirichlet_Sign_i} Under the assumption \ref{(A2)}:
\begin{align}
&\int_{\Omega \cap \{ u\in(0,B)\} } \operatorname{div}\A(x)\cdot  \nabla u(x) H(u(x))\dee x \ge 0,\label{a)} \\
&\int_{\Omega \cap \{ u\in(0,B)\}} \|\nabla u(x)\|_{\A}^{2} h(u(x)) \dee x  
\le  \int_{\Omega \cap \{ u\in(0,B)\}} |Pu (x)||H(u(x))|\dee x. \label{ineqality111}
\end{align}
\end{enumerate}
\end{Theorem}

\begin{proof}
\ref{Theorem_Dirichlet_Sign_ii}  By our assumptions, $\partial\Omega$ is contained in the level set of $w:= \widetilde{H}(u)$. Let us note that $\nabla w=\nabla \widetilde{H}(u)= H(u)\nabla u$ is parallel to the outer normal vector $n(x)$ for $\sigma$ almost every $x\in \partial \Omega$,  i.e.~it is of the form $\alpha(x) n(x)$, with some $\alpha: \partial \Omega \to \R$:

Indeed, when $\partial\Omega\subseteq \Rn$ is $C^1$ manifold and $w\in C^1\left ( \overline{\Omega} \right )$, then the statement is known by the classical theory. We assume that  $\Omega$ has the Lipschitz boundary and $w\in W_0^{1,1}(\Omega)\cap W^{2,1}(\Omega)$, therefore our assumptions are weaker. However, this fact is rather known to the specialists, for the reader's convenience we enclose the precise argument in Appendix~\ref{section Appendix}.
Thus, for $\sigma$-almost every $x\in\Omega$, it holds
$$
\alpha(x) = \langle \nabla w(x), n(x)\rangle =\lim_{t\to 0_+} \frac{w(x-tn(x))-w(x)}{-t}
= \lim_{t\to 0_+} \frac{w(x-tn(x))}{-t}\le 0,  
$$
and consequently
$$
n(x)^T \A(x)\nabla (\widetilde{H}(u)) = n(x)^T \A(x) \big ( \alpha(x) n(x)  \big ) = \alpha(x) \left ( n(x)^T \A(x) n(x) \right )  \le 0  \text{ $\sigma$-a.e.~on } \partial\Omega. 
$$
We deduce that $\Theta \leq 0$.

\ref{Theorem_Dirichlet_Sign_i} The estimate \eqref{a)} follows from the identity (denoting as $n_i(x)$ the $i$-th coordinate of the outer normal vector $n(x)$)

\begin{eqnarray*}
 \int_{\Omega \cap \{ u\in(0,B)\} } \operatorname{div}\A(x)\cdot  \nabla u H(u)\dee x &=& \int_{\Omega}\sum_{i,j\in \{ 1,\dots ,n\} } \frac{\partial a_{i,j}}{\partial x_j}\frac{\partial\widetilde{H}(u)}{\partial x_i}\dee x \\
  =\sum_{i,j\in \{ 1,\dots ,n\} } \int_{\partial\Omega} \frac{\partial a_{i,j}}{\partial x_j} n_i(x) \widetilde{H}(u)\dee \sigma  
&-& \sum_{i,j\in \{ 1,\dots ,n\} }\int_{\Omega} \frac{\partial^2 a_{i,j}}{\partial x_i\partial x_j } \widetilde{H}(u) \dee x \\
&\stackrel{\text{\ref{(u-TildeH)}} }{=}& -\int_\Omega \operatorname{div}^{(2)}\A(x) \widetilde{H}(u) \dee x \stackrel{\ref{(A2)}}{\ge} 0.
\end{eqnarray*}

Finally, the inequality \eqref{ineqality111} follows directly from \ref{Theorem_Dirichlet_Sign_ii}, \eqref{a)}, and \eqref{identity1}.
\end{proof}

\subsection{Analysis of the condition \ref{(u-TildeH)}}\label{u-tildeh}

The following statement contributes to the analyis of the condition \ref{(u-TildeH)}. It shows in paticular that if it holds, then the trace of $u$ is well presciribed   via \eqref{valueatpoint}, even though we are not assuming that $u\in W^{1,1}(\Omega)$ and such trace can be ``infinity''.

\begin{Theorem}\label{RemAn(u-TildeH-1)}
Let the condition  \ref{(u-TildeH)} holds. Then there is precisely one $T\in [0,B]$, allowing also $B=\infty$, such that $u\equiv T$ in $\partial\Omega$, where the trace of $u$ is prescribed  via \eqref{valueatpoint}.
\end{Theorem}

\begin{proof}
    We recall that $\widetilde{H}$ is non-negative, strictly convex, and continuous on $(0,B)$, whence it is monotone near the endpoints and so the limit
    \begin{equation} \label{RemAn(u-TildeH-1):E1}
        \lim_{\substack{t \to T \\ t \in (0, B)}} \widetilde{H}(t)
    \end{equation}
    exists for every $T \in [0, B]$ (recall that $B$ is allowed to be infinite).

    We first observe that there is necessarily some $T \in [0,B]$ such that the limit \eqref{RemAn(u-TildeH-1):E1} is zero. This follows from our assumption that $\widetilde{H}(u)\equiv 0$ on $\partial \Omega$, where the trace is given by \eqref{valueatpoint}, since we get for a.e.~$x \in \partial \Omega$
    \begin{equation*}
       0 =  \limsup_{r\to 0_+} \frac{1}{|\Omega\cap B(x,r)|}\int_{\Omega\cap B(x,r)} \widetilde{H}(u(y)) \dee y \geq \inf_{t \in I} \widetilde{H}(t) \geq 0,
    \end{equation*}
    where we recall that $u$ has values in $I$ and $(0,B) \subseteq I \subseteq [0,B]$. The existence of the desired $T$ now follows from the fact that the interval $[0,B]$ is compact (in the case $B = \infty$ we may interpret it as the one-point compactification); this of course means that this $T$ may be infinite.

    Next, we observe that there is only one $T \in [0,B]$ such that the limit \eqref{RemAn(u-TildeH-1):E1} is zero, which follows from the strict convexity and non-negativity of $\widetilde{H}$.

    Finally, we observe that we have for a.e.~$x \in \partial \Omega$ that
    \begin{equation*}
        \begin{split}
            0 &=  \limsup_{r\to 0_+} \frac{1}{|\Omega\cap B(x,r)|}\int_{\Omega\cap B(x,r)} \widetilde{H}(u(y)) \dee y \\
            &\geq \liminf_{r\to 0_+} \frac{1}{|\Omega\cap B(x,r)|}\int_{\Omega\cap B(x,r)} \widetilde{H}(u(y)) \dee y  \geq 0,
        \end{split}
    \end{equation*}
    so the limit superior is actually a limit. Hence we may apply the Jensen integral inequality to get
    \begin{equation} \label{RemAn(u-TildeH-1):E2}
        \begin{split}
            0 &=  \lim_{r\to 0_+} \frac{1}{|\Omega\cap B(x,r)|}\int_{\Omega\cap B(x,r)} \widetilde{H}(u(y)) \dee y \\
            &\geq \lim_{r\to 0_+} \widetilde{H} \left( \frac{1}{|\Omega\cap B(x,r)|}\int_{\Omega\cap B(x,r)} u(y) \dee y \right)  \geq 0,
        \end{split}
    \end{equation}
    where the integrals
    \begin{equation*}
        \frac{1}{|\Omega\cap B(x,r)|}\int_{\Omega\cap B(x,r)} u(y) \dee y
    \end{equation*}
    are always defined, finite or not, because $u$ is non-negative (if the integral is infinite for a given $r$ then we interpret $\widetilde{H}(\infty) = \lim_{t \to \infty} \widetilde{H}(t)$). It remains only to observe that this means that we also have
    \begin{equation*}
        \lim_{r\to 0_+} \frac{1}{|\Omega\cap B(x,r)|}\int_{\Omega\cap B(x,r)} u(y) \dee y = T,
    \end{equation*}
    where $T \in [0, B]$ is the unique point where the limit \eqref{RemAn(u-TildeH-1):E1} is zero, because otherwise we could construct a sequence of radii violating \eqref{RemAn(u-TildeH-1):E2}.
\end{proof}

\subsection{Opial-type inequalities}\label{opialtype}
Before we present another simplification of the inequality \eqref{ineqality1}, let us focus on certain variants of Opial-type inequalities, which seem interesting by themselves and, to the best of our knowledge, are not known. The analysis presented here is independent of Section~\ref{SectionBasicResults}. 

\smallskip
\noindent
We will now exploit a certain property of $h$: \ref{(h-H-TildeH)}, as introduced in Section~\ref{assumptions}.

\begin{Theorem}[Opial-type inequality for $\widetilde{H}(u)$ vanishing on the boundary]\label{Opial}~\\ Assume 
\ref{(Omega)}, \ref{(h)},  \ref{(I-H)},  
\ref{(h-H-TildeH)}, $u\in W^{1,1}_{\rm loc}(\Omega)$, $u(x)\in I$ a.e., and 
$\widetilde{H}(u)\in W_0^{1,1}(\Omega)$.
Then
\begin{equation}
\begin{aligned}
\int_{\Omega \cap \{ u\in (0,B)\}} \mathcal{G}_H(u) \dee x &=
\int_{\Omega \cap \{ u\in (0,B)\}} |\mathcal{T}_{H}(u)|^2 h(u)\dee x\\
&\le   C_PC_{\widetilde{H}} \int_{\Omega \cap \{ u\in (0,B)\}} \|\nabla u \| |\mathcal{T}_{H}(u)| h(u) \dee x, \label{opial1}\\
\end{aligned}
\end{equation}
\begin{equation}
\int_{\Omega \cap \{ u\in (0,B)\}} \|\nabla u\| |\mathcal{T}_{H}(u)| h(u)\dee x \le  (C_PC_{\widetilde{H}})^2 
\int_{\Omega \cap \{ u\in (0,B)\}} \|\nabla u\|^2 h(u)\dee x,\label{opial2}
\end{equation} 
where $C_P$ is the constant in the Poincar\'e inequality \eqref{poincare} and $C_{\widetilde{H}}$ is as in \ref{(h-H-TildeH)}. \\
Moreover, the integrals involved in \eqref{opial1} converge.
\end{Theorem}

\begin{proof}
In the case when $I=(0,B)$, we have by \ref{(h-H-TildeH)} that
$(\mathcal{T}_{H}(u))^2 h(u) = \frac{H^2(u)}{h(u)}\le C_{\widetilde{H}}|\widetilde{H}(u)|$. Moreover, we assume $\widetilde{H}(u) \in W_0^{1,1}(\Omega)$. 
Therefore, we can apply the Poincar\'e inequality \eqref{poincare} to get
\begin{eqnarray*}
\int_{\Omega } |\mathcal{T}_{H} (u) |^2 h(u)\dee x &=&
\int_{\Omega} \frac{H^2(u)}{h(u)}\dee x \le C_{\widetilde{H}}\int_{\Omega} |\widetilde{H}(u)|\dee x 
\stackrel{\eqref{poincare}}{\le} C_{\widetilde{H}}C_P
\int_{\Omega}\|\nabla\widetilde{H}(u)\|\dee x \\
&=&C_{\widetilde{H}} C_P\int_{\Omega} \|\nabla u\||H(u)|\dee x ,
\end{eqnarray*}
which yields \eqref{opial1} for this case.
We have used the ACL property (see Theorem~\ref{nikos} and Lemma~\ref{lcomp}), ensuring that the chain rule $\nabla{\widetilde{H}(u)} = H(u)\nabla u$ holds for almost every $x\in\Omega$. Note that $\widetilde{H}$ is Lipschitz on every segment $[\alpha,\beta]$ where $0<\alpha<\beta< B$.

For larger interval  $I$, we modify the above inequalities to 
\begin{equation*} \label{opial:E1}
\begin{split} 
\int_{\Omega\cap \{ u\in (0,B)\} } |\mathcal{T}_{H} (u) |^2 h(u)\dee x &= \int_{\Omega \cap \{ u\in (0,B)\}} \frac{H^2(u)}{h(u)}\dee x \le C_{\widetilde{H}}\int_{\Omega \cap \{ u\in (0,B)\}} |\widetilde{H}(u)|\dee x \\
&\le C_{\widetilde{H}}\int_{\Omega} |\widetilde{H}(u)|\dee x 
\stackrel{\eqref{poincare}}{\le}  C_{\widetilde{H}}C_P \int_{\Omega}\|\nabla\widetilde{H}(u)\|\dee x \\ 
&=  C_{\widetilde{H}} C_P\int_{\Omega \cap \{ u\in (0,B)\} } \|\nabla u\| |H(u)|\dee x .
\end{split}
\end{equation*}
Last equality holds, because when $T\in I\setminus (0,B)$, then for almost every point of the level set $\Omega_T:=\{ x\in\Omega : u(x)=T \}= \{ x\in\Omega : \widetilde{H}(u(x))=\widetilde{H}(T) \}$ we have $\nabla \widetilde{H}(u)=0$. We have thus proved \eqref{opial1} for any proposed $I$. Note that all terms in \eqref{opial1} are finite, because $\lVert \nabla u \rVert |H(u)|= \lVert \nabla\widetilde{H}(u) \rVert \in L^1(\Omega)$.

As for \eqref{opial2}, we  use the Cauchy--Schwarz inequality to obtain
$$
\begin{aligned}
&\int_{\Omega\cap \{ u\in (0,B)\}} \|\nabla u\||\mathcal{T}_{H}(u)| h(u) \dee x \\
&\le \left( \int_{\Omega \cap \{ u\in (0,B)\} } \|\nabla u\|^2 h(u) \dee x    \right)^{\frac{1}{2}}  \left( \int_{\Omega \cap \{ u\in (0,B)\}} |\mathcal{T}_{H}(u)|^2 h(u)\dee x    \right)^{\frac{1}{2}} =: \mathcal{A}^{\frac{1}{2}}\mathcal{B}^{\frac{1}{2}}.
\end{aligned}
$$

This together with \eqref{opial1}  gives
$\mathcal{B}\le C_{\widetilde{H}}C_P \mathcal{A}^{\frac{1}{2}}\mathcal{B}^{\frac{1}{2}}$, where $\mathcal{B}$ is finite by the arguments above.
Moreover, $\mathcal{A}$ can be assumed to be finite without loss of generality, as otherwise \eqref{opial2} holds trivially. Now, it suffices to rearrange this inequality.
\end{proof}

\subsection{Simplification based on Opial-type inequalities for
$\widetilde{H}(u)$ vanishing on the boundary}
Using the Opial-type inequalities from Theorem~\ref{Opial}, we obtain the following statement, which, under some extra assumptions, allows us to remove the term $\int_{\Omega\cap \{u\in (0, B)\} }\mathcal{G}_{H}(u)\dee x$ in the inequality \eqref{ineqality1}.
In particular, we assume that $\widetilde{H}(u)\equiv 0$ on the boundary of the domain.

\begin{Theorem}[inequalities for $\widetilde{H}(u)$ vanishing on the boundary]\label{simplification2}
 Assume that \ref{(G)}, \ref{(I-H)}, and \ref{(u-I-H)} hold, $\widetilde{H}$ is the given antiderivative  of $H$ which satisfies  \ref{(h-H-TildeH)}, and   $\widetilde{H}(u)\equiv 0$ on $\partial\Omega$. Define
 \begin{equation*}
 \Gamma  := c_{\A}^{-1}C_P^3C_{\widetilde{H}}^3, 
 \end{equation*}
 where $C_P$ is as in \eqref{poincare}, $C_{\widetilde{H}}$ is as in \ref{(h-H-TildeH)}, $c_{\A}$ is as in \ref{(A1)}. Then the following statements hold:

\begin{enumerate}
\item We have
\begin{eqnarray}\label{gh-1}
\int_{\Omega\cap \{u\in (0,B)\}  }
 \mathcal{G}_{H}(u)  \dee x &\le & \Gamma
\int_{\Omega \cap \{u\in (0,B)\} } \|\nabla u\|_{\A}^{2} h(u)\dee x .
\end{eqnarray}

\item When 
$$\kappa := \| \operatorname{div} \A\|_{L^\infty (\Omega)}c_{\A}^{-1}C_P^2C_{\widetilde{H}}^2$$ satisfies $0 < \kappa < 1$ then we also have
\begin{align}
 \int_{\Omega \cap \{u\in (0,B)\}  } \|\nabla u\|_{\A}^{2} h(u)\dee x  
&\le   \frac{1}{1-\kappa}\left( \int_{\Omega \cap \{u\in (0,B)\} } |Pu||H(u)| \dee x +\Theta \right),  \label{ineqality2}\\
\text{ where }\ \Theta &:= \int_{\partial\Omega }  n(x)^T \A(x) \nabla \widetilde{H}(u)\dee \sigma  (x) .\nonumber
\end{align}
\end{enumerate}
Furthermore, all the integrals involved in \eqref{gh-1} and \eqref{ineqality2} converge.
\end{Theorem}

\bigskip
\noindent
\begin{proof} The  inequality \eqref{gh-1} follows from \eqref{opial1} and \eqref{opial2} in Theorem~\ref{Opial} and from \eqref{gh}.
Moreover, both terms involved there are finite, by Theorem~\ref{LemmaBoundedUcoinfty}.
We are left with the proof of  \eqref{ineqality2}. 
For this, let us introduce a modified version of the notation used in the proof of Theorem~\ref{LemmaBoundedU}:
$$\mathcal{I}:=\left(\int_{\Omega \cap \{ u\in(0,B)\}} \|\nabla u\|_{\A}^{2} h(u)\dee x\right)^{\frac{1}{2}} ,\ \ \ \mathcal{J}:= \int_{\Omega \cap \{ u\in(0,B)\} } |Pu||H(u)| \dee x,$$ and we know that $\mathcal{I}$ and $\mathcal{J}$ are finite.
We  have
\begin{equation*}
\begin{aligned}
&\left|\int_{\Omega \cap \{ u\in(0,B)\} } \operatorname{div} \A\cdot \nabla u \;  H (u) \dee x \right|\le 
\| \operatorname{div} \A\|_{L^\infty(\Omega)} \int_{\Omega \cap \{ u\in(0,B)\} } \| \nabla u \| \;  |H (u)| \dee x \\
&\stackrel{\eqref{opial2}}{\le} \| \operatorname{div} \A\|_{L^\infty(\Omega)} (C_PC_{\widetilde{H}})^2 \int_{\Omega \cap \{ u\in(0,B)\}} \|\nabla u\|^2 h(u)\dee x  \\
&
\stackrel{\ref{(A1)}}{\le} \| \operatorname{div} \A\|_{L^\infty(\Omega)} (C_PC_{\widetilde{H}})^2 c_{\A}^{-1} \int_{\Omega \cap \{ u\in(0,B)\}} \|\nabla u\|_{\A}^2 h(u)\dee x  =\kappa \mathcal{I}^2.
\end{aligned}
\end{equation*}
Then we use the identity \eqref{identity1}, to get
$$
\mathcal{I}^2\le \mathcal{J} + \kappa \mathcal{I}^2 + \Theta \text{ with }  \kappa <1,
$$
where $\mathcal{I},\mathcal{J},\Theta$ are finite. Let us recall that the finiteness of $\Theta$ follows from the assumption $\nabla\widetilde{H}(u) = \nabla u \;H(u)\in W^{1,1}(\Omega,\R^n)$ as in Theorem~\ref{LemmaBoundedU}.
This, after rearranging, gives the desired conclusion. 
\end{proof}

\section{Links with the literature in probability and potential theory}\label{probabilityka}
In this section, we discuss the links of our results with probability theory and potential theory, particularly with the analysis of generators of analytic semi-groups.

\subsection{Integral chain-rule-type upper bound}\label{chain-sect}

Let us first focus on the pointwise chain rule for the composition $u\mapsto \widetilde{H}(u)$, under the suitable regularity assumptions.

When we deal with the function $u$ of one variable, we have
\begin{equation}\label{one-var-chain}
(\widetilde{H}(u))'= \widetilde{H}'(u)u' = Q(\widetilde{H})u',\text{ where } \widetilde{H}\mapsto Q(\widetilde{H}):=\widetilde{H}' \text{ is a linear operator}.
\end{equation}
\\
For second-order differential operators we have:
\begin{align}
(\widetilde{H}(u))''&= (\widetilde{H}'(u)u')'= \widetilde{H}''(u)(u')^2
+ \widetilde{H}'(u)u''  &\text{for a function of one variable}, \nonumber \\
P(\widetilde{H}(u))&=  \widetilde{H}'(u)Pu + 
\widetilde{H}''(u)\|\nabla u\|_{\A}^2 &\text{in general, where $P$ is as in \eqref{pe}}.\label{chainrule}
\end{align}
There is no linear operator $T \colon \widetilde{H}\mapsto T(\widetilde{H})$ such that the right-hand side above would be of the form $T(\widetilde{H})u^{''}$ or $T(\widetilde{H})Pu$, respectively. Hence, for a second-order elliptic operator,  the chain rule like  $P\widetilde{H}(u)=T(\widetilde{H}) Pu$ does not hold in the pointwise sense.

On the other hand, it follows from our statements that,  under suitable assumptions (as in either Theorem~\ref{structureA} or Theorem~\ref{simplification2}), the non-linear term $\widetilde{H}''(u)\|\nabla u\|_{\A}^2$   in 
 \eqref{chainrule} is integrally dominated by the linear term $\widetilde{H}'(u)Pu$, where the linearity applies to $u$. Consequently,
 the following integral chain-rule-type  upper bound holds: 
\begin{equation}\label{chainrule-1}
\begin{split}
    \int_\Omega |P(\widetilde{H}(u))|\dee x \, &\stackrel{\mathmakebox[\widthof{=}]{\eqref{chainrule}}}{\le} \, \int_\Omega |\widetilde{H}'(u)Pu|\dee x + \int_\Omega |(\widetilde{H})''(u)|\|\nabla u\|_{\A}^2 \dee x\\
    &\stackrel{\mathmakebox[\widthof{=}]{\textup{Thms}\  \ref{structureA},\, \ref{simplification2}} }{\lesssim} \hspace{0.7cm} \int_\Omega  |\widetilde{H}'(u)Pu|\dee x \\
    &= \int_\Omega  |Q(\widetilde{H}) Pu|\dee x,\  \text{ where } Q(\widetilde{H})=  (\widetilde{H})'\ \text{as in \eqref{one-var-chain}}.
\end{split}
\end{equation}
For a similar conclusion in the one-variable setting, we refer to the inequalities from \cite{km} and \cite{akjp}, where $Pu=u^{''}$.

Deep results, concerning general variants of the pointwise chain rule formula \eqref{chainrule}, dealing with the infinitesimal generators of the diffusion processes in place of the elliptic operator $P$, can be found in \cite[Lemma~1, p.~179]{bakry-emery}. 

In the setting of the fractional $\frac{\alpha}{2}$-Laplacian, there exist some known pointwise analogues of \eqref{chainrule-1}; namely, the so-called pointwise C{\' o}rdoba--C{\' o}rdoba inequalities (see~\cite{CordobaCordoba03}):
$$
\Lambda^\alpha (\phi (f))(x) \le \phi^{'}(f)\Lambda^\alpha f(x),
$$
where $\phi\in C^1(\mathbb{R})$ is convex, $f$ belongs to the Schwartz space $\mathcal{S}(\mathbb{R}^n)$, and $\Lambda^\alpha = (-\Delta)^{\frac{\alpha}{2}}$ is the fractional $\frac{\alpha}{2}$-Laplacian. For an extension to the setting of compact Riemannian manifolds, we refer the reader to \cite{CordobaMartinez15}. The possibility of generalising the C{\' o}rdoba-C{\' o}rdoba inequalities to non-divergent non-local operators has been suggested in \cite[Section~2.1]{CaffarelliSire17}. Therefore, we propose the following problem about the extension of our results to non-local operators.

\begin{oq}[extension to non-local operators] {\rm It would be interesting  to obtain an (as general as possible) integral chain-rule-type upper bound similar to \eqref{chainrule-1} for non-local operators, such as the fractional Laplacian or stable-like operators, instead of the classical elliptic operator $P$ as defined in \eqref{pe}.}
\end{oq}

\subsection{The identities 
arising in  the analysis of generators of analytic semi-groups in  $L^p(\R^n)$}   \label{SectionMetafune}

When estimating  the angle of analyticity of semi-group generators in the $L^p$ setting, the following identity appeared (see \cite{metafone-spina} for this and other motivations, derivation, and related references):
\begin{equation}\label{metafonespinaeq}
\int_{\R^n}
u |u|^{p-2}\Delta u \dee x = - (p-1)
\int_{\R^n}
|u|^{p-2} \| \nabla u\|^2 \dee x .
\end{equation}
Here we also focus on the following identity obtained in \cite[Theorem 4.1]{metafone-spina}, which is more general: 
\begin{equation}\label{metafonespinaeq1}
\int_{\R^n}
g(u) |g(u) |^{p-2} \Delta u \dee x 
= - (p-1) \int_{\R^n}
\| \nabla u\|^2 g^{*} (u)| g(u) |^{p-2} \chi_{g(u)\neq 0} \dee x ,
\end{equation}
where  $g : \R \rightarrow \R$  is a uniformly Lipschitz and monotone function such
that $g(0) = 0$,  $g^*$ is such a monotone Borel function that $g^* = g'$ almost everywhere,
$u \in W^{2,p}(\R^n )$, $1 < p < \infty$. 

Omitting the precise arguments,
 if we assume that $g\in C^1(\R)$ is non-decreasing, put 
$h(s)=(p-1) g'(s)|g(s)|^{p-2}\chi_{g(s)\neq 0}$, assume for simplicity that $u$ is compactly supported in $\R^n$ and non-negative,  take $I:= [0,\infty)$, $H(s):= |g(s)|^{p-1}$ in \ref{(I-H)}, and finally use \eqref{identity1}, we retrieve a variant of \eqref{metafonespinaeq1} in a specific case.

\smallskip
\noindent
We consider the following open problems.

 \begin{oq}[extension of the results to $\Omega=\R^n$]
 \rm In \eqref{metafonespinaeq} and \eqref{metafonespinaeq1}, one deals with identities on $\R^n$, while we deal with the identities on its sufficiently regular bounded domains. 
 It would be of interest to obtain some variants of our identities with general operator $P$,  holding on $\Omega= \R^n$.
 \end{oq}

 \begin{oq}[extension without sign condition]\rm 
The equations \eqref{metafonespinaeq} and \eqref{metafonespinaeq1} do not require function 
$u$ to be non-negative (see also \cite[Lemma 2.1]{metafone-spina}), as we do assume. Therefore, we believe that an extension of our results to the case of not necessarily non-negative functions would be possible. The results in this direction in the case of the analysis on subsets of $\R$ can be found in \cite{akjp}.
\end{oq} 

\subsection{Douglas formulae, Sobolev--Bregman forms, and harmonic extensions in the potential theory}

\subsubsection*{Douglas formulae}

When working on the Plateau problem, J. Douglas
discovered the following identity  \cite{douglas}:
\begin{equation}\label{dougl-ineq}
\int_{B(0,1)} \|\nabla u\|^2 \dee x =\frac{1}{8\pi}
\int_0^{2\pi}\int_0^{2\pi}\frac{(g(\eta)-g(\xi))^2}{\sin^2 ((\xi-\eta)/2)}\dee \eta\dee \xi ,
\end{equation}
where $u: \R^2\supseteq B(0,1)\rightarrow \R$ is a harmonic function, $u\equiv g$ on $\partial\Omega$. 

Douglas formulae later found several extensions and applications to the theory of Markov processes and subordinated Dirichlet forms (see e.g. \cite{chen, fukushima2, jacob}),   analysis of elliptic operators generating analytic semi-groups \cite{pazy}, harmonic analysis and potential theory
(\cite{bogdan}, \cite[Theorem 3.1]{verbitsky}). Let us focus on some selected results.

\subsubsection*{Extensions dealing with harmonic functions on bounded domains} 
In the recent article \cite{bogdan} one finds the following formula (see \cite[Theorem~15]{bogdan}, together with the notation ${\cal{E}}_{{\Omega}}^p[u]$ and ${{\cal{H}}^p_{{\partial\Omega}}}[u]$ established at the beginning of \cite[Section~3]{bogdan}):
\begin{equation}\label{dougl-ext}
\begin{split}
I&:= p(p-1)\int_\Omega \|\nabla u\|^2 |u(x)|^{p-2} \dee x \\
&= \frac{p}{2} \int_{\partial \Omega}\int_{\partial \Omega} \left( u^{\langle p-1\rangle}(z)- u^{\langle p-1\rangle}(w)\right)\left( u(z)- u(w)\right)
\gamma_\Omega (z,w) \dee\sigma (z) \dee\sigma (w) \\
&\hspace{3cm}-p\int_\Omega \Delta u\,  u^{\langle p-1\rangle} \dee x
+\frac{p}{2}\int_\Omega\Delta u(x) P_{\Omega} [u^{\langle p-1\rangle}](x) \dee x,
\end{split}
\end{equation}
where: 
\begin{itemize}
\item $p\in [2,\infty)$, $\Omega\subseteq \R^n$ is a bounded domain of class $C^{1,1}$, $u\in C^2(\overline{\Omega})$ (non-negativity is not required), $\dee\sigma$ is the surface measure on the boundary of {$\Omega$}, and  $a^{\langle \kappa \rangle} = |a|^{\kappa}{\rm sgn}\, a$;
 \item $\gamma_\Omega (z,w)$ is the Feller kernel (see e.g.~\cite[Lemma~1]{Zhao}), i.e.~the inner normal derivative of the Poisson kernel on $\partial \Omega$ (it is known that $\gamma_D(z,w)\approx |z-w|^{-n}$);
\item 
$P_\Omega [v]=\int_{\partial \Omega}v(z)P_\Omega(x,z)\dee z$ is the Poisson integral of the boundary data $v$, i.e.~the harmonic extension of $v$ to the inside of $\Omega$. In particular, $P_\Omega [u^{\langle p-1\rangle}]$ is harmonic inside $\Omega$ and we have $P_\Omega [u^{\langle p-1\rangle}] \equiv u^{\langle p-1\rangle}$ on $\partial\Omega$. 
\end{itemize}
Let us now consider our identity \eqref{identity}, with $P=\Delta$ and $h(u) = p(p-1)u^{p-2}$, which requires slightly different assumptions on the involved function $u$ (namely, the assumptions of Theorem~\ref{LemmaBoundedUcoinfty}):
\begin{eqnarray} 
I = p(p-1) \int_{\Omega } \|\nabla u\|^{2} u(x)^{p-2}\dee x  &=& -p\int_{\Omega} \Delta u \; u^{p-1} \dee x  
     +\Theta, \text{ where} \label{identity-xx}\\
\Theta &=& \int_{\partial\Omega}  n(x)^T  \nabla  ( u^p)\dee \sigma (x). \label{identity-xx-Theta}
\end{eqnarray}
Combining \eqref{dougl-ext} and \eqref{identity-xx} (under the appropriate intersection of the respective assumptions), we obtain the following Douglas-type representation of $\Theta$:
\begin{equation}\label{formulae-douglis-theta}
\begin{split}
\Theta = \frac{p}{2} \int_{\partial \Omega}\int_{\partial \Omega} \left( u^{ p-1}(z)- u^{p-1}(w)\right) &\left( u(z)- u(w)\right)
\gamma_\Omega (z,w) \dee\sigma(z) \dee\sigma(w)\\
&+\frac{p}{2}\int_\Omega\Delta u(x)  P_{\Omega}[u^{p-1}] \dee x .
\end{split}
\end{equation}
Observe that the formula \eqref{identity-xx-Theta} for $\Theta$ 
involves $\nabla(u^p)= pu^{p-1}\nabla u$ on $\partial\Omega$, in the appropriate points of differentiability. In particular, it requires not only the knowledge of $u$ on $\partial \Omega$, but also of $\nabla u$. In contrast, the right-hand side of \eqref{formulae-douglis-theta} involves only values of $u$ on $\partial\Omega$ and the values of $\Delta u$ inside $\Omega$. Moreover, the first term on the right-hand side of \eqref{formulae-douglis-theta} is the so-called Sobolev--Bregman form related to $\partial\Omega$ (see e.g.~\cite{BogdanGrzywny23} or \cite{BogdanJakubowski22}); we find it rather surprising that this form has emerged from our computations in this manner.

We believe that the formula \eqref{formulae-douglis-theta} is interesting as of itself. Therefore, we address the following problem:

 \begin{oq}[Douglas type representation of $\Theta$]\rm
What are the required conditions (as weak as possible) on the involved function $u$ and the domain $\Omega$ in order for the formula \eqref{formulae-douglis-theta} to hold?
  \end{oq}

\appendix

\section{Appendix}\label{section Appendix}

For the reader's convenience, we include some more detailed arguments for some of our statements.

The following remark provides details for Remark \ref{rem-ii_abbrv}.

\begin{Remark}[about the condition \ref{(I-H)}]\label{rem-ii}\rm
$H$ is strictly increasing on $(0,B)$, because $H^{'}=h>0$, and thus $\widetilde{H}$ is strictly convex. In particular, $H$ can change sign at most once in $(0,B)$, whence $\widetilde{H}$ is strictly monotone near the endpoints of $(0,B)$ and the limits $\lim_{s\to 0_+}\widetilde{H}(s)$ and $\lim_{s\to B_-}\widetilde{H}(s)$ exist, finite or not. If either of the limits is finite, then $\widetilde{H}$ can be continuously extended to the corresponding endpoint by setting the value to be the respective limit. Furthermore, the information about the finiteness of the limits gives us additional information about the integrability of $H$ near the respective endpoints of $I$. 

Namely, if we put $A$ as either $0$ or $B$, we get
\begin{equation*}
    \widetilde{H}(s) - \lim_{\substack{t \to A \\ t \in (0,B)}} \widetilde{H}(t) = \lim_{\substack{t \to A \\ t \in (0,B)}}  \int_t^s H(\tau) \dee \tau = \int_A^s H(\tau) \dee \tau.
\end{equation*}
The last integral is always defined, finite or not, because $H$ is increasing, and so it does not change sign on some neighbourhood of $A$. The last equality holds by the Monotone Convergence Theorem. Hence, $H$ is integrable near $0$ or $B$ if and only if the respective limit $\lim_{s\to 0_+}\widetilde{H}(s)$ or $\lim_{s\to B_-}\widetilde{H}(s)$ is finite and $\widetilde{H}$ can be extended to this endpoint. 

Finally, if $H$ is integrable near either of the endpoints, then we may for example construct an antiderivative of $H$ for $s \in I$ as either the Hardy or conjugate-Hardy transform of $H$, i.e.
\begin{align*}
    \widetilde{H}(s)&:= \int_0^s H(\tau) \dee \tau,
    &\widetilde{H}(s):= -\int_s^B H(\tau) \dee \tau.
\end{align*}
depending on whether $H$ is integrable near $0$ or near $B$.

On the other hand, we never require that the functions $H$ and $h$ can be extended to $I$.
\end{Remark}

Next, we present the proof of the following lemma that was required for the proof of Theorem~\ref{structureA}. We note, that while the result is classical for more regular domains, the case of Lipschitz domains is apparently not as standard. Hence, we provide the proof for the sake of completeness.

\begin{Lemma}\label{lem-app1}
If $\Omega\in C^{0,1}$ and  $v\in W_0^{1,1}(\Omega)\cap W^{2,1}(\Omega)$, then for $\sigma$ almost every $y\in\partial\Omega$ the vector $\nabla v$ as in \eqref{valueatpoint}
 is parallel to the normal vector to $\partial \Omega$.  
\end{Lemma}

\begin{proof}
{\sc{Step 1}} Reduction argument and some geometric objects.\\
{\sc Reduction argument.}
Using the localisation argument and rigid motions (see e.g. \cite[Section~6]{KufnerJohn77}), we can assume that $\Omega$
is a subgraph of some Lipschitz function $ \Phi : (0,1)^{n-1}\rightarrow \R $:
\begin{eqnarray*}
\Omega = \left \{ (x^{'},x_n): x^{'}\in (0,1)^{n-1}, x_n\in (0,\Phi (x^{'})) \right\},
\end{eqnarray*}
and prove that $\nabla v$ is $\sigma$-a.e.~parallel to the normal vector to the ``graph part'' of the boundary, which is the Lipschitz $(n-1)$-dimensional submanifold  
$$
M:= \{ (x^{'},x_n): x^{'}\in (0,1)^{n-1}, x_n = \Phi (x^{'})\} \subseteq   \R^n.
$$

\noindent
{\sc Geometric objects.}
For given $$y= (y^{'},y_n)=  \Psi(y^{'}):= (y^{'},\Phi(y^{'}))\in M,$$  tangent  space to $M$ at $y=(y^{'},y_n)$ is 
\begin{equation*}\label{tan}
T_yM:= \{ (x^{'}, \nabla \Phi (y^{'}) \cdot x^{'}) : x^{'}\in \R^{n-1}  \}={\rm span}\left \{ w_i:= \left (e_i,\frac{\partial\Phi (y^{'})}{\partial y_i} \right ),\ i=1,\dots, n-1  \right \} . 
\end{equation*}
As $\Phi$ is only Lipschitz, by the Rademacher Theorem (see \cite[Theorem~9.14]{Leoni17}), tangent space is defined for $y$  in the set $\mathcal{B}\subseteq M$ of
full Hausdorff measure $\sigma$, 
but perhaps not everywhere. 
Note that
$$\mathcal{B}=\{ 
y= \Psi(y^{'}):  \Phi\ \hbox{\rm is differentiable at}\ y^{'}  
\}=: \Psi (\mathcal{C}),\ {\rm where}\ \sigma ((0,1)^{n-1} \setminus \mathcal{C})=0. $$

\noindent
One defines Lebesgue and Sobolev spaces, see e.g. \cite[Definition 6.7.2]{KufnerJohn77},
\begin{eqnarray*}
L^p(M) &=& \{ f: M\rightarrow \R: f\circ \Psi \in L^p((0,1)^{n-1})\} ,\\
W^{1,1}(M) &=& \{ f: M\rightarrow \R: f\circ \Psi \in W^{1,1}((0,1)^{n-1})\} . 
\end{eqnarray*}
Trace Theorem (see e.g.~\cite[Theorem~6.7.8]{KufnerJohn77}, where (A.1), (A.2) are a special cases) 
implies that, when $u\in W^{2,1}(\Omega)$ and $\Omega\in C^{0,1}$, then  
\begin{eqnarray}\label{trace-w2}
\nabla u\in L^1(M,\R^n) && {\rm and}\  \| \nabla u\|_{L^1(M,\rn)}\lesssim \|\nabla  u\|_{W^{1,1}(\Omega)} \lesssim \|  u\|_{W^{2,1}(\Omega)},\\
 u\in W^{1,1}(M) && {\rm and}\  \|  u\|_{W^{1,1}(M)} \lesssim \|  u\|_{W^{2,1}(\Omega)}.\label{trace-w2-1}
\end{eqnarray}
Although \eqref{trace-w2} and \eqref{trace-w2-1} are known, \eqref{trace-w2-1} seems not so obvious at first glance in the setting of Lipschitz domains. For reader's convenience we explain them by direct arguments. Indeed, \eqref{trace-w2} is the simplest classical Trace Theorem in the setting of $W^{1,1}(\Omega)$, which holds wherever $\Omega$ is Lipschitz (see e.g.~\cite[Theorem~6.4.1]{KufnerJohn77}; we provide further information in Section~\ref{SectionPropSobolevFunc}). To show \eqref{trace-w2-1} it suffices to note that  chain rule 
\begin{eqnarray}\label{chainrule-ap}
\frac{\partial}{\partial x_i}(u\circ \Psi)(x)=\nabla u(\Psi(x))\cdot \frac{\partial\Psi}{\partial x_i}(x)\ \ {\rm in}\ \  D^{'}((0,1)^{n-1})
\end{eqnarray}
holds. This is obvious when $u\in C^1\left ( \overline{\Omega} \right)$. For $u\in W^{2,1}(\Omega)$ it follows from the $C^1\left ( \overline{\Omega} \right )$-case by an approximation argument, since even the set $C^\infty \left ( \overline{\Omega} \right )$ is dense in $W^{2,1}(\Omega)$ for $\Omega \in C^{0,1}$ (see e.g.~\cite[Therorem~11.35]{Leoni17}).
In particular $\nabla u(\Psi(x))\in L^1((0,1)^{n-1})$ and $\frac{\partial\Psi}{\partial x_i}(x)\in L^{\infty}((0,1)^{n-1})$. 
Therefore the right-hand side in \eqref{chainrule-ap} belongs to $L^1((0,1)^{n-1})$ and so the property (A.2) holds.

\noindent
{\sc Step 2.} We complete the proof.\\
As mentioned above, $C^\infty \left ( \overline{\Omega} \right)$ is dense in $W^{2,1}(\Omega)$, so let us consider the sequence $\{ v_k\}\subseteq C^\infty \left ( \overline{\Omega} \right)$ such that $v_k\to v$ in $W^{2,1}(\Omega)$ as $k\to\infty$. We do not have guarantee that $v_k\in W_0^{1,1}(\Omega)$. 
However,  we can deduce, using also 
  \eqref{trace-w2} that
\begin{eqnarray*}
\nabla v_k\stackrel{k\to\infty}{\rightarrow } \nabla v \ {\rm in}\ W^{1,1}(\Omega ,\rn)\ \ \ {\rm and}\ \ \ 
\nabla v_k \stackrel{k\to\infty}{\rightarrow} \nabla v\ {\rm in}\ L^1(M, \rn) . \label{sladnabrzegu}
\end{eqnarray*}
We can additionally assume, after eventually passing to the subsequence, that 
\begin{equation*}\label{setc}
\nabla v_k (y) \stackrel{k\to\infty}{\rightarrow}\nabla v (y) \ \ \hbox{for every} \ y\in\mathcal{D}\subseteq M, \ {\rm where}\ \sigma (M\setminus \mathcal{D})=0. 
\end{equation*} 
We note that $\Psi^{-1}(\mathcal{D})$ is also of full measure in $(0,1)^{n-1}$. This follows from the area formula, since $\Psi$ is Lipschitz and regular a.e., meaning that its Jacobian is positive a.e.~on $(0,1)^{n-1}$. Indeed,
\begin{equation*}
    J(D\Psi) = \sqrt{\det( D\Psi^T D\Psi))} = \sqrt{\det (\operatorname{Id}_{n-1} + \nabla \Phi \otimes \nabla\Phi)} = \sqrt{1 + \Vert \nabla \Phi \rVert^2 } > 0. 
\end{equation*} 

\noindent
At the same time, by \eqref{trace-w2-1}, $v_k\in W^{1,1}(M)$ and 
$$v_k\stackrel{k\to\infty}{\rightarrow} v\ {\rm  in}\  W^{1,1}(M).$$
Equivalently,
$$
v_k\circ \Psi \stackrel{k\to\infty}{\rightarrow} v\circ \Psi\ {\rm in}\ W^{1,1}((0,1)^{n-1}).
$$
However, $v\in W_0^{1,1}(\Omega)$, so we have $v\circ \Psi \equiv 0$ on $(0,1)^{n-1}$, and consequently
$$
v_k\circ \Psi \stackrel{k\to\infty}{\rightarrow} 0 \ {\rm in}\ W^{1,1}((0,1)^{n-1}). 
$$
After eventually passing to an appropriate subsequence, we can assume that there is a set of full measure $\mathcal{F}\subseteq (0,1)^{n-1},$ such that $\sigma ((0,1)^{n-1}\setminus \mathcal{F})=0$ and it holds
\begin{equation}\label{convtozero}
(v_k\circ \Psi) (y^{'}) \stackrel{k\to\infty}{\rightarrow} 0\ {\rm and}\ \nabla (v_k\circ \Psi)(y^{'}) \stackrel{k\to\infty}{\rightarrow} 0\ \hbox{\rm for every}\ y^{'}\in \mathcal{F}.
\end{equation}
To simplify the presentation, we will write the gradient as a row vector for the rest of the proof. Then, we have for every $y^{'}\in \mathcal{C}\cap \mathcal{F}\cap \Psi^{-1}(\mathcal{D})$ (which is of full measure in $(0,1)^{n-1}$)
\begin{eqnarray*}
\nabla (v_k\circ \Psi )(y^{'}) \stackrel{(a)}{=} [(\nabla v_k) (\Psi (y^{'}))] D \Psi (y^{'})
 \stackrel{(b)}{\rightarrow} [(\nabla v)(\Psi (y^{'}))] D \Psi (y^{'})\in \mathbf{R}^{n-1},
\end{eqnarray*}
where (a) holds at every point $y^{'}\in \mathcal{C}$ by the chain rule involving smooth $v_k$ and Lipschitz $\Psi$, while (b) holds for $y^{'}\in \Psi^{-1} (\mathcal{D})$. The limit is zero because of \eqref{convtozero}.

Therefore, when  $y=(y^{'},\Phi(y^{'}))\in M$ and $y^{'}\in 
\mathcal{C}\cap \mathcal{F}\cap \Psi^{-1}(\mathcal{D})\subseteq (0,1)^{n-1}$, we have 
\begin{eqnarray*}
\begin{pmatrix}
0,\dots, 0
\end{pmatrix}
=
\nabla v(y)
\begin{pmatrix}
1&0&\dots &0\\
0&1&\dots &0\\
\vdots&\vdots&\dots &\vdots\\
0&0&\dots &1 \\
\frac{\partial\Phi (y^{'})}{\partial y_1} & \frac{\partial\Phi (y^{'})}{\partial y_2} & \cdots & \frac{\partial\Phi (y^{'})}{\partial y_{n-1}}
\end{pmatrix}
=
\begin{pmatrix}
\nabla v(y) \cdot w_1, \dots,
 \nabla v(y) \cdot w_{n-1}
\end{pmatrix}.
\end{eqnarray*}
As $$ \left \{w_i  \right\}_{i=1,\dots,n-1} = \left \{ \left (e_i,\frac{\partial \Phi (y^{'})}{\partial y_i} \right)  \right\}_{i=1,\dots,n-1}$$ form a basis of $T_yM$ and $\nabla v(y) \cdot w_i=0$ for every $i$,  the tangential part of $\nabla v(y)$ is zero. To complete the proof, it suffices to note that the set of such admissible $y\in \Psi ( \mathcal{C}\cap \mathcal{F}\cap \Psi^{-1}(\mathcal{D} ))$ is of full $\sigma$-measure in $M$.

\end{proof}

\section{Acknowledgements}

A.K. was supported in part by National Science Centre (Poland), grant Opus\\
2023/51/B/ST1/02209. D.~P.~was supported by the Grant schemes at Charles University,  reg.~No.~CZ.02.2.69/0.0/0.0/19\_073/0016935, the grants no.~P201-18-00580S, P201/21-01976S, and P202/23-04720S of the Czech Science Foundation, and Charles University Research program No.~UNCE/SCI/023. T.~R.~was supported by the grant GA\v{C}R 20-19018Y  of the Czech Science Foundation.

\smallskip

The idea to derive multiplicative inequalities with an elliptic operator was first proposed to the first author by Patrizia Donato, to whom we are grateful for suggesting an interesting problem. We would like to thank  Krzysztof Bogdan, Petr Girg, and Artur Rutkowski for the helpful information about related results in PDEs and potential theory. We are grateful to the University of Warsaw and the Faculty of Mathematics, Informatics and Mechanics for the semester-long visit of D.P. and T.R. in 2020, which started the collaboration but has to be unfortunately switched to online due to pandemics.

\section{Data availability statement}

We do not analyse or generate any datasets, because our work is theoretical in nature.

\bibliographystyle{dabbrv}
\bibliography{biblio}

\end{document}